\documentclass[11pt,reqno]{amsart}
\usepackage{fullpage}

\usepackage{graphicx}
\usepackage[draft]{hyperref}
\usepackage{amsmath,amsopn,amssymb,amsfonts,stmaryrd,setspace}
\usepackage{verbatim}
\usepackage{amsthm}
\usepackage{mathtools}
\usepackage{color}
\usepackage{enumitem}
\usepackage[framemethod=TikZ]{mdframed}
\usepackage{bbm}
\usepackage{mathrsfs}
\usepackage{booktabs}
\usepackage{caption}
\usepackage{bm}
\usepackage{tensor}
\usepackage{cleveref}
\usepackage{ytableau}
\usepackage{centernot}

\newcommand{\IN}{{\mathbb N}}
\newcommand{\IQ}{{\mathbb Q}}

\newcommand{\sgn}{\mbox{sgn}}

\newcommand{\SL}{\mathrm{SL}}

\newcommand{\N}{\mathbb N}
\newcommand{\C}{\mathbb C}

\theoremstyle{plain}
\newtheorem{thm}{Theorem}[section]
\newtheorem{cor}[thm]{Corollary}
\newtheorem{lem}[thm]{Lemma}
\newtheorem{prop}[thm]{Proposition}
\newtheorem{conj}[thm]{Conjecture}

\newtheorem*{rem}{Remark}

\theoremstyle{definition}
\newtheorem{example}[thm]{Example}
\newtheorem*{example*}{Example}
\newtheorem{defn}[thm]{Definition}
\newtheorem*{defn*}{Definition}

\numberwithin{equation}{section}

\newcommand{\pmat}[1]{\left( \smallmatrix #1 \endsmallmatrix \right)}

\renewcommand{\sgn}{\textnormal{sgn}}

\def\lp{\left(}
\def\rp{\right)}

\def\a{\alpha}

\def\z{\zeta}

\def\n{\nu}

\def\t{\tau}

\def\a{\alpha}

\def\z{\zeta}

\def\n{\nu}

\def\t{\tau}

\renewcommand{\sgn}{{\rm sgn}}
\newcommand{\R}{\mathbb R}
\newcommand{\Z}{\mathbb Z}

\newcommand{\lrb}[1]{\left(#1\right)}
\newcommand{\lrsb}[1]{(#1)}
\newcommand{\lrB}[1]{\left[#1\right]}

\def\bar{\overline}

\setlist[itemize]{noitemsep, topsep=0pt}

\makeatletter
\newcommand{\vast}{\bBigg@{2}}
\newcommand{\Vast}{\bBigg@{5}}
\makeatother

\renewcommand{\pmod}[1]{\ \left( \mathrm{mod} \, #1 \right)}
\newcommand{\Pmod}[1]{\ ( \mathrm{mod} \, #1 )}

\newcommand{\QA}{\IQ\langle A\rangle}
\newcommand{\qsh}{\ast_\diamond}

\begin{document}
\title[Limiting behaviour and modular completions of MacMahon-like $q$-series]{Limiting behaviour and modular completions\\ of MacMahon-like $q$-series}

\author[K. Bringmann]{Kathrin Bringmann}
\address{University of Cologne, Department of Mathematics and Computer Science,
	Weyertal 86-90,
	50931 Cologne, Germany}
\email{kbringma@uni-koeln.de}
\email{wcraig@uni-koeln.de}
\email{j.w.ittersum@uni-koeln.de}
\email{bpandey@uni-koeln.de}

\author[W. Craig]{William Craig}

\author[J-W. van Ittersum]{Jan-Willem van Ittersum}

\author[B. Pandey]{Badri Vishal Pandey}

\makeatletter
\@namedef{subjclassname@2020}{%
	\textup{2020} Mathematics Subject Classification}
\makeatother

\copyrightinfo{}{}

\subjclass[2020]{11F03, 11A25, 11F50, 11F11}
\keywords{Divisor sums, hook lengths, partitions, $q$-multiple zeta values, quasimodular forms}

	\begin{abstract}
		Recently, MacMahon's generalized sum-of-divisor functions were shown to link partitions, quasimodular forms, and $q$-multiple zeta values. In this paper, we explore many further properties and extensions of these. Firstly, we address a question of Ono by producing infinite families of MacMahon-like functions that approximate the colored partition functions (and indeed other eta quotients).  We further explore the MacMahon-like functions and discover new and suggestive arithmetic structure and modular completions.
	\end{abstract}

\maketitle

\section{Introduction and statement of results}\label{sec:intro}
	A {\it partition} of $n \in\N_0$ is a list $\lambda = \lp \lambda_1, \lambda_2, \dots, \lambda_\ell \rp$ of positive integers in weakly decreasing order whose sum of parts, denoted by $|\lambda|$, is $n$. We write $\lambda \vdash n$ to denote that $\lambda$ is a partition of~$n$. In this paper, we consider a relationship between partitions and divisor sums. For $k\in\N_0$, we consider the \textit{sum-of-divisors functions}
	\[
		\sigma_k(n) := \sum_{d \mid n} d^k.
	\]
	These play a central role in number theory. The special case $\sigma_1(n)$ has the generating function
	\[
		\mathcal{A}_1(q) := \sum_{n\ge1} \sigma_1(n)q^n = \sum_{n\ge1} \frac{q^n}{(1 - q^n)^2}.
	\]
	MacMahon {\cite{M1921}} broadly generalized this construction using partition theory to study divisor sums. In particular, on \cite[pp. 303, 309]{M1921} he defined for $a\in\N$ the $q$-series
	\begin{align*}
		\mathcal{A}_a(q) &:= \sum_{1\leq n_1 < \cdots < n_a}\frac{q^{n_1 + \cdots + n_a}}{(1 - q^{n_1})^2 \cdots (1 - q^{n_a})^2}.
	\end{align*}
	MacMahon proposed many interesting directions of study for these functions. Andrews and Rose \cite[Corollary 4]{AR2013} proved that the family of functions $\{\mathcal{A}_a\}_{a\in\N}$ belong to the algebra of quasimodular forms. More recently, Amdeberhan--Andrews--Tauraso \cite{AAT2023} and Amdeberhan--Ono--Singh \cite{AOS2023} considered $\mathcal{A}_a$ and showed many congruences for its Fourier coefficients.

	We next recall \cite[Theorem~1.1]{AOS2023}, which shows that the family of quasimodular forms $\mathcal{A}_a$ form a sequence of approximations for the $3$-colored partition function, which produces a hook length formula for small coefficients of $\mathcal{A}_a$. The key ingredients for the proof of \Cref{T:AOS} are the results of Andrews--Rose \cite[Corollary 2]{AR2013} on the functions $\mathcal{A}_a$ and the Nekrasov--Okounkov hook lengths formula \cite{NO} which relates sums over hook numbers of partitions to powers of the \textit{Dedekind eta-function} ($q:=e^{2\pi i \t}$ throughout) $\eta(\t) := q^{\frac{1}{24}} \prod_{n=1}^\infty \lp 1 - q^n \rp$.
	\begin{thm}[Amdeberhan--Ono--Singh]\label{T:AOS}
		Let $a \in \N$.
		\begin{enumerate}[leftmargin=*]
			\item[\rm (1)] We have
			\begin{align*}
			q^{- \frac{a(a+1)}{2}} \mathcal{A}_a(q) = \prod_{n \geq 1} \dfrac{1}{\lp 1 - q^n \rp^3} + O\lp q^{a+1} \rp.
			\end{align*}
			\item[\rm (2)] If $n \leq a$, then we have
			\begin{align*}
			\operatorname{coeff}_{\lrB{q^{n+ \frac{a(a+1)}{2}}}} \mathcal{A}_a(q) = \sum_{\lambda \vdash n} \prod_{h \in \mathcal H(\lambda)} \lp 1 + \dfrac{2}{h^2} \rp = \sum_{\lambda \vdash n} \prod_{s=1}^{n-a} \dfrac{(m_s+1)(m_s+2)}{2},
			\end{align*}
			where $\mathcal{H}(\lambda)$ is the multiset of hook lengths in $\lambda$ (see \Cref{Prelims}) and $m_s = m_s(\lambda)$ denotes the number of occurrences of $s$ in $\lambda$.
		\end{enumerate}
	\end{thm}

	In this paper, we consider an open problem posed by Ono \cite{O2024}. For $k \in\N$, the \textit{$k$-colored partition function} $p_k(n)$ counts the number of partitions of $n$, where parts are allowed to take one of $k$ colors. For example, $p(n) := p_1(n)$ counts the number of partitions of $n$. Note the generating function
	\begin{align*}
		\sum_{n \geq 0} p_k(n) q^n = \prod_{n \geq 1} \dfrac{1}{\lp 1 - q^n \rp^k}.
	\end{align*}
	Ono \cite{O2024} suggested that there should be a natural family of quasimodular forms generalizing $\mathcal{A}_a$ whose coefficients approximate other colored partition functions, or more broadly, other powers of~$\eta$. In this paper, we produce these natural families. In order to state our theorems, we define 
	for $a,r,s \in \N$ and $k\in \Z$ the $q$-series
		\begin{align*}
			\mathcal{A}_{a,k,r}(q) :=& \sum_{1\leq n_1 < n_2 < \dots < n_a} \dfrac{q^{r(n_1 + \dots + n_a)}}{\lp 1 - q^{n_1} \rp^k \cdots \lp 1 - q^{n_a} \rp^k} =: \sum_{n \geq 0} c_{a,k,r}(n) q^n ,\\
			\mathcal{B}_{a,k,r,s}(q) :=& \sum_{1\leq n_1 < n_2 < \dots < n_a}  \dfrac{q^{r\lp n_1^{2}+\dots + n_a^2\rp  +s\lp n_1+\dots+n_a\rp}}{\lp 1 - q^{n_1} \rp^k \cdots \lp 1 - q^{n_a} \rp^k} =: \sum_{n \geq 0} d_{a,k,r,s}(n) q^n.
		\end{align*}
	For $k\geq1$, the function $\mathcal{A}_{a,k,r}$ appears in the literature of $q$-analogues of multiple zeta values, such as in~\cite{BK16}. Note that $\mathcal{A}_a = \mathcal{A}_{a,2,1}$. The leading term in the asymptotic expansion of $\mathcal{A}_{a,k,r}$ and $\mathcal{B}_{a,k,r,s}$ as $q\to 1$ is a multiple of the same multiple zeta value; the function $\mathcal{B}_{a,k,r,s}$ seems, however, to not have been previously studied in the context of $q$-analogues of multiple zeta values.

	We first give a generalization of \Cref{T:AOS} by showing that the functions $\mathcal{A}_{a,k,r}$ and $\mathcal{B}_{a,k,r,s}$ each form a family of approximations of certain eta quotients (depending on $k$ and $r$) as $a$ varies, which includes colored partitions as a special case. Then, using the Nekrasov--Okounkov hook length formula, we give a formula for $n$-th coefficient of $\mathcal{A}_{a,k,r}$ and $\mathcal{B}_{a,k,r,s}$ for small $n$.

	\begin{thm} \label{T: Main Theorem}
		Let $a,r,s \in \N$ and $k \in \Z$.
		\begin{enumerate}[label*=\rm(\arabic*),leftmargin=*]
			\item We have
			\begin{align*}
				q^{- \frac{ra(a+1)}{2}} \mathcal{A}_{a,k,r}(q) &= \prod_{n \geq 1} \dfrac{1}{\lp 1 - q^{rn} \rp \lp 1 - q^n \rp^k} + O\lp q^{a+1} \rp, \\
				q^{- \frac{a (a+1) ( r+2 a r+3 s)}{6} } \mathcal{B}_{a,k,r,s}(q) &= \prod_{n \geq 1} \dfrac{1}{\lp 1 - q^n \rp^k} + O\lp q^{a+1} \rp.
			\end{align*}
			\item If $n \leq a$, then we have
			\begin{align*}
				c_{a,k,r}\left(n+\frac{ra(a+1)}{2}\right) = \sum_{m=0}^{\left\lfloor\frac{n}{r}\right\rfloor} p(m) \sum_{\lambda\vdash(n-rm)} \prod_{h\in\mathcal H(\lambda)}\lrb{1+\dfrac{k-1}{h^2}}.
			\end{align*}
			\item If $n \leq a$, then we have
			\begin{align*}
				c_{a,k,1}\left(n+\frac{a(a+1)}{2}\right) = d_{a,k+1,r,s}\left( n + \dfrac{a(a+1)( r + 2ar+3s)}{6} \right) = \sum_{\lambda \vdash n} \prod_{h \in \mathcal H(\lambda)} \lp 1 + \dfrac{k}{h^2} \rp.
			\end{align*}
		\end{enumerate}
	\end{thm}
We also study the behavior of the coefficients of $\mathcal{A}_{a,k,r}$ and $\mathcal{B}_{a,k,r,s}$ for varying $k$ and $r$ (and $s$). As a second result, we have the following.
\begin{thm}\label{thm:main2}
There exist polynomials $P_{a,\ell,n}$ and $P_{a,\ell,m,n}$ of degree at most $\ell$ such that
\begin{align*}
\mathcal{A}_{a,k,r}(q) &= \sum_{\ell\geq 0} \sum_{n\ge \frac{a(a+1)}2} P_{a,\ell,n}(k) q^{nr+\ell}, \quad
\mathcal{B}_{a,k,r,s}(q) = \sum_{\ell\geq 0} \sum_{n\ge \frac{a(a+1)}2} \sum_{\frac{n^2}{a}\leq m\leq n^2} P_{a,\ell,m,n}(k) q^{mr+ns+\ell} .
\end{align*}
\end{thm}
As a direct corollary we obtain:
\begin{cor}\label{cor:main2}
There exist polynomials $\mathcal{P}_{a,r,n}$ and $\mathcal{P}_{a,r,s,n}$ of degree $n-\frac{1}{2}ra(a+1)$ and $n-{\frac{1}{2}a (a+1) (r + 2ar + 3s)}$, respectively, such that
\begin{align*}
\mathcal{A}_{a,k,r}(q) &= \sum_{n\ge0} \mathcal{P}_{a,r,n}(k) q^n, \quad
\mathcal{B}_{a,k,r,s}(q) = \sum_{n\geq 0} \mathcal{P}_{a,r,s,n}(k) q^n.
\end{align*}
\end{cor}

\begin{example*}
We have
\begin{align*}
\mathcal{A}_{2,k,r}(q) &= q^{3r} + k q^{3r+1} + \frac{k(k+3)}{2}q^{3r+2} +q^{4r} + k q^{4r+1} +\frac{k\lrb{k+1}}{2} q^{4r+2} + 2 q^{5r} +kq^{5r+1} + \ldots,
\\
\mathcal{A}_{2,k,1}(q) &= q^3+(k+1)q^4+\frac{(k+1)(k+4)}{2}q^5+\frac{k^3+12 k^2+11 k+12}{6}q^6 + \ldots, \\ 
 \mathcal{B}_{2,k,r,s}(q) &=  q^{5 r+3 s}+k q^{5 r+3 s+1}+\frac{k(k+3)}{2}q^{5 r+3 s+2} +\frac{k(k^2+9k+2)}{6}q^{5 r+3 s+3}  +q^{10 r+4 s}\\
						&\hspace{1cm} +k q^{10 r+4 s+1} + \frac{k(k+1)}{2} q^{10 r+4 s+2}  + \ldots,\\
\mathcal{B}_{2,k,1,1}(q) &= q^8+k q^9+\frac{k(k+3)}2q^{10}+\frac{k\left(k^2+9 k+2\right)}6 q^{11}+ \frac{k^4+18 k^3+35 k^2+18 k}{24}  q^{12}+\ldots. 
\end{align*}
\end{example*}

	In line with the work of Andrews--Rose \cite{AR2013} and Amdeberhan--Ono--Singh \cite{AOS2023}, we prove that many examples of the $\mathcal{A}_{a,k,r}$ are quasimodular forms. Note that the functions $\mathcal{A}_{a,2r,r}$ were considered by Okounkov in the context of computing generating functions of Hilbert schemes \cite{Oko14}.

	\begin{thm}\label{T:Quasimodular}
		If $a,r \in \N$, then $\mathcal{A}_{a,2r,r}$ is a quasimodular form of (mixed) weight $2ar$.
	\end{thm}
For other choices of $k$, the function $\mathcal{A}_{a,k,r}$ fails to be quasimodular, but nevertheless can be written as a polynomial in Eisenstein series. Note, however, that in this case both the (quasi)modular Eisenstein series of even weight and the odd weight Eisenstein series appear, so in all cases $\mathcal{A}_{a,k,r}$ can be represented using divisor sums. For example, we have
		\begin{align*}
			\mathcal{A}_{2,4,2} &= \frac{1}{504} G_4^2+\frac{1}{720} G_6-\frac{1}{36} G_4G_2-\frac{53}{8640}G_4+\frac{1}{72} G_2^2+\frac{293}{60480}G_2+\frac{649}{3225600}, \\
			\mathcal{A}_{2,4,3} &= \frac{1}{504} G_4^2+\frac{1}{720} G_7-\frac{1}{12} G_4G_3+\frac{1}{8} G_3^2+\frac{1}{18} G_4G_2-\frac{1}{144} G_6+\frac{1}{72} G_5-\frac{1}{6} G_3G_2 \\
			&\qquad +\frac{1}{18} G_2^2-\frac{23}{8640}G_4-\frac{7}{320} G_3+\frac{493}{30240}G_2+\frac{5707}{9676800}.
		\end{align*}

Due to its connections with quasimodular forms, it can be shown that the coefficients of $\mathcal{A}_{a,2r,r}$ satisfy many congruences. In particular, we can prove the following corollary to \Cref{T:Quasimodular} along with known results on $p$-adic modular forms and quasimodular forms.
\begin{cor} \label{C:Congruences}
	Let $a,r,m \in \N$. There are infinitely many non-nested arithmetic progressions $An+B$ such that
	\begin{align*}
		c_{a,2r,r}\lp An+B \rp \equiv 0 \pmod{m}.
	\end{align*}
\end{cor}
\noindent In \Cref{conj:congr}, we present a large family of congruences for $c_{a,k,r}(n)$ and $d_{a,k,r,s}(n)$.

We now turn to modularity properties of $\mathcal{B}_{a,k,r,s}$; for more precise results, see \Cref{T:Trans} and \Cref{T:Ell}. We call a function $\widehat f(\tau,\overline\tau)$ a \emph{completion} of $f(\tau)$ if\footnote{Here and throughout, we consider $\tau$ and $\overline\tau$ as independent variables. If clear from the context we also just write $\widehat f(\tau)$ instead of $\widehat f(\tau,\overline\tau)$.}  $\lim_{\overline\tau\to-i\infty}\widehat f(\tau,\overline\tau)=f(\tau)$.
\begin{thm}\label{T:Completing}Let $k,r,s\in \N$.
\begin{enumerate}[label*=\rm(\arabic*),leftmargin=*]
			\item The function $\mathcal{B}_{a,k,r,s}$ is a polynomial in $\mathcal{B}_{k,r,s}:=\mathcal{B}_{1,k,r,s}$ of degree at most $a$.
			\item For $0\leq s \leq k$, the functions $\mathcal{B}_{k,r,s} + \mathcal{B}_{k,r,k-s}$ are a mixed weight linear combination of {\rm(a)}~functions admitting a modular completion and {\rm(b)} powers of $q$ times theta functions.
			\item If additionally $k=s+1$, part {\rm(b)} in the above linear combination vanishes.
\end{enumerate}
	\end{thm}

\begin{example}\label{E:Psi}
Letting $\tau=u+iv$ throughout, we have
\begin{equation*}
\mathcal{B}_{2,2,1,1}(q) = \frac{1}{2}\mathcal{B}_{2,1,1}^2(q) -\frac{1}{2} \mathcal{B}_{4,2,2}(q), \qquad
2\mathcal{B}_{2,1,1}(q) = {\left( \frac{\pi}{6} + \frac{1}{v}\right)i\psi_{-1}^+(\tau) - \frac{1}{2\pi i} \psi_1^+ }(\tau)
\end{equation*}
for some functions $\psi_k^+(\tau)$ which admit a modular completion $\widehat \psi_k(\tau,\overline\tau)$ of weight $k+1$ (note $k=2=s+1$). 
We have
\begin{align*}
\widehat\psi_{-1}(\tau,\overline\tau) &= \psi_{-1}^+(\tau) = -\frac1{2\pi i}, \\
 \widehat\psi_1(\tau,\overline\tau) &= \psi_1^+(\tau) - 8\pi iT(\tau) \sum_{n\in\Z+\frac12} n\left(\sgn(n)-E\left(n\sqrt{2v}\right)\right) q^{-\frac{n^2}2} + \frac{8\sqrt2i}{\sqrt v} |T(\tau)|^2,
\end{align*}
where
\begin{equation*}
	T(\tau) := \sum_{n\ge0} q^{\frac18(2n+1)^2} = \frac{\eta(2\tau)^2}{\eta(\tau)}
\end{equation*}
is a theta function of weight $\frac12$, and for $x\in \R$ we define
\begin{align}
\label{eq:E}
E(x):=2\int_{0}^x e^{-\pi t^2} dt.
\end{align}
\end{example}


	The remainder of the paper is organized as follows. In \Cref{Prelims}, we provide certain preliminaries on hook numbers, quasimodular forms, $q$-multiple zeta values, Zwegers' higher Appell functions, and generalized Stirling numbers. In \Cref{Proofs}, we prove the results regarding the limiting behaviour and the coefficient structure of $\mathcal{A}_{a,k,r}$ and $\mathcal{B}_{a,k,r,s}$ for varying parameters. We also show that $\mathcal{A}_{a,2r,r}$ are quasimodular and that their coefficients satisfy many congruences. In \Cref{sec:completing}, we prove \Cref{T:Completing} using higher Appell functions. In \Cref{sec:open}, we sketch how the arguments extend to variations of the functions~$\mathcal{A}_{a,k,r}$ and~$\mathcal{B}_{a,k,r,s}$ and provide several conjectures.

	\section*{Acknowledgements}

	{The authors thank Ken Ono for bringing their attention to these questions and Henrik Bachmann and Annika Burmester for several remarks. This work was funded by the European Research Council (ERC) under the European Union's Horizon 2020 research and innovation programme (grant agreement No. 101001179) and by the SFB/TRR 191 “Symplectic Structure in Geometry, Algebra and Dynamics”, funded by the DFG (Projektnummer 281071066 TRR 191).

	\section{Preliminaries} \label{Prelims}

	\subsection{Hook lengths}
	For a partition $\lambda = \lp \lambda_1, \dots, \lambda_\ell \rp$ of $n$, its {\it Young diagram} is a collection of $n$ boxes arranged in $\ell$ left-adjusted rows, with row $j$ containing $\lambda_j$ boxes. Abusing notation, we identify $\lambda$ with its Young diagram. For each cell in $\lambda$, say at row $j$ and column $k$, the {\it hook} $H(j,k)$ is the collection of boxes including the given box, along with all boxes directly beneath or to the right of this box, such that the overall shape is generically an inverted $L$-shape. The {\it hook length} $h(j,k)$ is the number of boxes contained in the hook. \Cref{Figure} gives an example of the Young diagram of $\lambda = (4,3,2)$ with each box labeled with its hook number. For a partition $\lambda$, we let $\mathcal{H}\lp \lambda \rp$ denote the multiset of all hook lengths in~$\lambda$.
	\begin{figure}[hbt!]
		\centering $$\begin{ytableau} 6&5&3&1 \\ 4&3&1 \\ 2&1  \end{ytableau}$$
		\caption{{\small Hook numbers of the partition $\lambda=(4,3,2)$}}
		\label{Figure}
	\end{figure}

	We also require the celebrated Nekrasov--Okounkov hook length formula (see (6.12) of~\cite{NO}).
	\begin{thm}[Nekrasov--Okounkov] \label{T: NO}
		For $z\in\C$ we have
		\begin{align*}
		\sum_{\lambda \in \mathscr P} q^{|\lambda|} \prod_{h \in \mathcal{H}\lp \lambda \rp} \lp 1 - \dfrac{z}{h^2} \rp = \prod_{n \geq 1} \lp 1 - q^n \rp^{z-1},
		\end{align*}
		where $\mathscr{P}$ is the set of all partitions.
	\end{thm}

\subsection{Modular forms and quasimodular forms}
For $k\in\N$ even, let
	\[
		E_k (\t) := 1-\frac{2k}{B_k} \sum_{n\geq 1} \sigma_{k-1}(n)q^n = 1-\frac{2k}{B_k} \sum_{m,n\geq 1} m^{k-1}q^{mn}
	\]
be the \emph{Eisenstein series of weight $k$}, where $B_{k}$ denotes the $k$-th Bernoulli number. If $k\ge4$, then $E_k$ is a modular form of weight $k$. For $k=2$, $E_2$ is a quasimodular form. That is, for $\begin{psmallmatrix}a&b\\c&d\end{psmallmatrix}\in\SL_2(\Z)$,
\begin{equation*}
	E_2\left(\frac{a\tau+b}{c\tau+d}\right) = (c\tau+d)^2 E_2(\tau) - \frac{6ic}\pi (c\tau+d).
\end{equation*}
More generally a quasimodular form can be written as a polynomial in $E_2$ with modular coefficients.
Note that $E_2$ has a completion
\begin{equation*}
	\widehat E_2(\tau) := \widehat E_2(\tau,\overline\tau) = E_2(\tau) - \frac3{\pi v}.
\end{equation*}
That is, for $\begin{psmallmatrix}a&b\\c&d\end{psmallmatrix}\in\SL_2(\Z)$, we have
\begin{equation*}
	\widehat E_2\left(\frac{a\tau+b}{c\tau+d}\right) = (c\tau+d)^2 \widehat E_2(\tau), \quad \lim_{\overline\tau\to-i\infty} \widehat E_2(\tau,\overline\tau) = E_2(\tau).
\end{equation*}
\subsection{Quasi-shuffle products and $q$-multiple zeta values}
Let $a\in\N_0$, $k_1,\ldots,k_a\in\N$, and $P_1,\ldots,P_a\in x\mathbb{Q}[x]$ with $P_j$ of degree at most $k_j$ for all $1\leq j\leq a$. Following \cite{BK20}, we define
	\begin{equation*}
		g_{d_1,\ldots,d_a;P_1,\ldots,P_a}(q) := \sum_{1\leq n_1<n_2<\cdots<n_a} \prod_{j=1}^a \dfrac{P_j(q^{n_j})}{(1-q^{n_j})^{k_j}}.
	\end{equation*}
We write $\mathcal{Z}_q$ for the $\mathbb{Q}$-algebra generated by these functions.\footnote{Often, authors allow $P_1,\ldots,P_a \in \mathbb{Q}[x]$ and only request that $P_a \in x\mathbb{Q}[x]$. Conjecturally, the corresponding $\mathbb{Q}$-algebras are isomorphic \cite[Conjecture 3.15]{BI22}.} Note that $\mathcal{A}_{a,k,r}$ is obtained by setting $k_1=\cdots=k_a=k$ and $P_j(x)=x^r$. The functions $g_{d_1,\ldots,d_a;P_1,\ldots, P_a}$ are examples of $q$-analogues of multiple zeta values. That is, assuming $k_a\geq 2$, we have \cite[p.~6]{BK20}
	\[
	\lim_{q\to 1} (1-q)^{k_1+\cdots+k_a} g_{k_1,\ldots,k_a;P_1,\ldots,P_a}(q) = \zeta(k_1,\ldots,k_a)\prod_j P_j(1),
	\]
with the \emph{multiple zeta value}
\[
\zeta(k_1,\ldots,k_a):=\sum_{1\leq n_1<n_2<\cdots<n_a} \prod_{j=1}^a n_j^{-k_j}.
\]
Write $\mathcal{Z}$ for the $\IQ$-algebra of these (real) numbers.  If all the polynomials $P_j$ are monomials, $P_j(x)=x^{r_j}$ for some $r_j\in\N$, then we write (abusing notation)
\[
g_{k_1,\ldots,k_a;r_1,\ldots,r_a}(q) := \sum_{1\leq n_1<n_2<\cdots<n_a} \prod_{j=1}^a \dfrac{q^{r_jn_j}}{(1-q^{n_j})^{k_j}}.
\]
Particular cases include the Schlesinger--Zudilin model of $q$-analogues \cite{Sch01, Zud15} if $r_j=k_j$ and the Bradley--Zhao model \cite{Bra05} if $r_j=k_j-1$.
\begin{rem}
These functions are intimately related to the theory of partitions {\upshape(}see \cite[Proposition~3.9 and Example~3.10]{BI22}{\upshape)}. Let $k_1,\ldots,k_a\in \N$ and $r_1,\ldots,r_a\in \N$ as above, let $\lambda \in \mathscr{P}$, and define
\[
H_{k_1,\ldots,k_a;r_1,\ldots,r_a}(\lambda) := \sum_{1\leq n_1<n_2<\cdots<n_a} \prod_{j=1}^a \binom{m_{n_j}(\lambda)+k_j-r_j}{m_{n_j}(\lambda)-r_j+1}.
\]
Then, under the $q$-bracket ${\IQ}^\mathscr{P} \to \IQ\llbracket q\rrbracket$ given by
\[ \langle f \rangle_q := \frac{\sum_{\lambda \in \mathscr{P}} f(\lambda) q^{|\lambda|}}{\sum_{\lambda \in \mathscr{P}}  q^{|\lambda|}},\]
one has
\[ \langle H_{k_1,\ldots,k_a;r_1,\ldots,r_a} \rangle_q = g_{k_1,\ldots,k_a;r_1,\ldots,r_a}(q).\]
\end{rem}

Note that (see equation~(1.19) of \cite{Bac21})
\[ E_k = 1-\frac{2k}{B_k} g_{k;\mathcal E_k}\]
for some family of polynomials $\mathcal E_k$, called the Eulerian polynomials. Generalizing the weight assigned to quasimodular forms, the space $\mathcal{Z}_q$ admits a weight filtration, and the \emph{weight} of $g_{k_1,\ldots,k_a;r_1,\ldots,r_a}$ is given as $k_1+\cdots+k_a$. Of particular interest is the algebra structure of $\mathcal{Z}_q$, that is, $\mathcal{Z}_q$ is an example of a quasi-shuffle algebra, defined as follows (see \cite{Hof00}).
Suppose that we have a set $A$, called the \emph{set of letters}, and an associative and commutative product \emph{diamond} $\diamond$ on $A$. Extending this product bilinearly to $\IQ A$ gives a commutative non-unital $\IQ$-algebra $(\IQ A, \diamond)$. We are interested in $\IQ$-linear combinations of multiple letters of $A$, i.e., in elements of $\QA$. Here and in the following we call monic monomials in $\QA$ \emph{words}. 

\begin{defn*}
Define the \emph{quasi-shuffle product}  $\qsh$ on $A$ as the $\IQ$-bilinear product which satisfies $1 \qsh w = w \qsh 1 = w$ for any word $w\in \QA$ and
    \begin{equation*}
        a w \qsh b t = a (w \qsh b t) + b (a w \qsh t) + (a \diamond b) (w \qsh  t)
    \end{equation*}
    for any letters $a,b \in A$ and words $w,t \in \QA$.
\end{defn*}

\begin{example*} Let $A=\IN$ and $a\diamond b=a+b$ for $a,b\in A$. The subalgebra of words not starting with the letter~$1$ surjects onto $\mathcal{Z}$. In particular, the usual product of real numbers on $\mathcal{Z}$ can be interpreted as a quasi-shuffle product. For example, we have
\[ \zeta(k_1)\zeta(k_2) = \zeta(k_1,k_2)+\zeta(k_2,k_1)+\zeta(k_1+k_2) .\]
\end{example*}
\begin{example*}
Let $A=\IN^2$ and $(d,r)\diamond (e,s) = (d+e,r+s)$. Then, the quasi-shuffle algebra on $A$ surjects onto the algebra $\mathcal{Z}_q$, e.g.,
\[ g_{k;r} g_{\ell;s} = g_{k,\ell;r,s} + g_{\ell,k;s,r} + g_{k+\ell;r+s}.\]
\end{example*}

As the main application of the theory of quasi-shuffle algebras, we recall the following proposition, which is equation~(32) in \cite{HI17}. For a word $a\in A$, let $a^{\circ n}, a^{\diamond n}$, and $a^{\qsh n}$ be the $n$-fold concatenation, diamond, and quasi-shuffle product, respectively, of $a$ with itself. Let
\[\exp_{\qsh}(aX) := \sum_{n\geq 0} a^{\qsh n} \frac{ X^n}{n!} .\]
\begin{prop}\label{prop:qsh} For $a\in A$ one has
\[ \exp_{\qsh}\left(\sum_{n\geq 1} \frac{(-1)^{n+1}}{n} X^n a^{\diamond n} \right) = \sum_{j\geq 0} a^{\circ j} X^j.\]
\end{prop}

\subsection{Multivariable Appell functions}\label{sec:Appell}
Define Zwegers' \emph{multivariable Appell function \cite{Zw}} for $\ell\in\N$
	\begin{equation}\label{eq:Appell}
		A_{\ell}\left(z_{1},z_{2};\tau\right):=\zeta_{1}^{\frac{\ell}{2}} \sum_{n\in\Z} \frac{(-1)^{\ell n} q^{\frac{\ell n(n+1)}{2}}\zeta_{2}^{n}}{1-\zeta_{1}q^{n}},
	\end{equation}
	where $\zeta_j:=e^{2\pi i z_j}$ for $j\in \{1,2\}$. To describe its modularity, we require the {\it Jacobi theta function}
	\begin{equation*} 
		\vartheta\lrb{z;\t}:=\sum_{n\in\Z+\frac{1}{2}} e^{2\pi i n\lrb{z+\frac{1}{2}}} q^{\frac{n^2}{2}},
	\end{equation*}
	and ($\tau=u+iv$, $z=x+iy$)
	\begin{equation*}
		R\lrb{z;\t} = R(z,\overline{z};\tau,\overline\tau):= \sum_{n\in\Z+\frac{1}{2}} \lrb{\sgn(n)-E\lrb{\lrb{n+\frac yv}\sqrt{2v}}}(-1)^{n-\frac{1}{2}} q^{-\frac{n^2}{2}}e^{-2\pi i n z},
	\end{equation*}
	where $E$ is given by \eqref{eq:E}.
	Recall the elliptic transformation ($a$, $b\in\Z$)
	\begin{equation}\label{E:Elliptic}
		\vartheta(z+a\tau+b;\tau) = (-1)^{a+b}q^{-\frac{a^2}{2}}\zeta^{-a} \vartheta(z;\tau), \quad \vartheta(-z;\tau)=-\vartheta(z;\tau)
	\end{equation}
	as well as the following transformations (see Proposition~1.9 in \cite{ZwTh})
\begin{equation}\label{eq:Rell}
R(z+1;\tau)=-R(z;\tau), \quad  R(z+\tau;\tau) = -\zeta q^{\frac12} R(z;\tau) + 2\zeta^{\frac12}q^{\frac34}, \quad R(-z;\tau)=R(z;\tau).
\end{equation}

	We have, by equation~(4.2) of~\cite{Zw}
		\begin{equation}\label{eq:multipletosingleappell}
			A_{\ell}(z_1,z_2;\tau) = \frac{1}{\ell}\zeta_1^{\frac{\ell-1}{2}} \sum_{j\pmod{\ell}} A_1\!\left(z_1,\frac{z_2+j}{\ell}+\frac{(\ell-1)\tau}{2\ell};\frac{\tau}{\ell}\right).
		\end{equation}
	We then define the completion\footnote{Note that here, the word ``completion'' is used in a wider sense than defined above. In the cases of interest in this paper, we have a completion in the above sense.} of $A_\ell(z_1,z_2;\tau)$
	\begin{align*}
		\widehat{A}_{\ell}(z_{1},z_{2};\tau) &= \widehat{A}_{\ell}(z_{1},z_{2},\overline{z_1},\overline{z_2};\tau,\overline\tau)\\
		&:=A_{\ell}(z_{1},z_{2};\tau) + \frac{i}{2} \sum_{j=0}^{\ell-1} e^{2\pi ijz_{1}} \vartheta\left(z_{2}+j\tau+\frac{\ell-1}{2};\ell\tau\right) R\left(\ell z_{1}-z_{2}-j\tau-\frac{\ell-1}{2};\ell\tau\right).
	\end{align*}
	We now recapitulate Theorem~4 of~\cite{Zw}, which gives the Jacobi transformation properties of $\widehat{A}_\ell$.
	\begin{lem}\label{L:Ahat}\hspace{0cm}
		\begin{enumerate}[wide,labelwidth=!,labelindent=0pt,label=\rm(\arabic*)]
			\item\label{I:Ahatsl2z} For $\pmat{a&b\\c&d}\in\SL_{2}(\Z)$, we have the transformation law
			\begin{equation*}
				\widehat{A}_{\ell}\left(\frac{z_{1}}{c\tau+d},\frac{z_{2}}{c\tau+d};\frac{a\tau+b}{c\tau+d}\right) = (c\tau+d)e^{\frac{\pi ic }{c\tau+d}\left(-\ell z_{1}^{2}+2z_{1}z_{2}\right)} \widehat{A}_{\ell}(z_{1},z_{2};\tau).
			\end{equation*}
			\item\label{I:Ahatell} For $n_{1},n_{2},m_{1},m_{2} \in\Z$, we have
			\begin{equation*}
				\widehat{A}_{\ell}(z_{1}+n_{1}\tau+m_{1},z_{2}+n_{2}\tau+m_{2};\tau) = (-1)^{\ell(n_{1}+m_{1})} \zeta_{1}^{\ell n_{1}-n_{2}} \zeta_{2}^{-n_{1}} q^{\frac{\ell n_{1}^{2}}{2} - n_{1}n_{2}} \widehat{A}_{\ell} (z_{1},z_{2};\tau).
			\end{equation*}
		\end{enumerate}
	\end{lem}
 Of particular interest is Zwegers' \emph{$\mu$-function}
	\begin{equation}\label{eq:mu}
		\mu(z_1,z_2;\tau) := \frac{A_1(z_1,z_2;\tau)}{\vartheta(z_2;\tau)},
	\end{equation}
	which he studied in his Ph.D. thesis. Using Proposition 1.4 (3), (4) of \cite{ZwTh}, we can show inductively:
	\begin{lem}\label{L:MuShift}
		For $n\in \Z$, we have
		\begin{multline*}
			\mu(z_1,z_2-n\tau;\tau)
			= (-1)^n q^\frac{n^2}2 e^{2\pi in(z_1-z_2)} \mu(z_1,z_2;\tau) \\ + ie^{-\pi i\sgn(n)(z_1-z_2)} q^{-\frac18-\frac{|n|}2} \sum_{j=1}^{|n|} (-1)^j e^{2\pi ij\sgn(n)(z_1-z_2)} q^{-\frac{j(j-1)}2+j|n|}.
		\end{multline*}
	\end{lem}

\subsection{Generalized Stirling numbers and derivatives}	\label{sec:stirling}
Consider the generalized {\it Stirling numbers} $S_{n,\ell}(a,b,c)$ (see\footnote{There $S_{n,\ell}(a,b,c)$ was denoted by $S(n,\ell,b,a,c)$.} (1) and Theorem~1 of \cite{HS98}) uniquely defined via the recurrence
	\begin{equation}\label{eq:Stir}
		S_{n+1,\ell}(a,b,c) = \lrb{a\ell-bn+c}S_{n,\ell}(a,b,c) +S_{n,\ell-1}(a,b,c),
	\end{equation}
with initial conditions $S_{0,0}\lrb{a,b,c}=1,$ and $S_{n,\ell}\lrb{a,b,c}= 0$ if $\ell<0$ or $\ell>n$. The usual (signed) Stirling number $s(n,\ell)$ is the special case $s(n,\ell)=S_{n,\ell}(0,1,0)$ as we have $s(n+1,\ell)=-n s(n,\ell)+s(n,\ell-1)$ with $s(0,0)=1$ and $s(n,\ell)=0$ if $\ell<0$ or $\ell>n$. We let
	\begin{equation}\label{eq:defalphabeta}
		\alpha_c(n,\ell) := S_{n,\ell} \left(0,1,c\right), \quad
		\beta_c(n,\ell):= S_{n,\ell}(1,0,c),
\end{equation}
which by \eqref{eq:Stir} satisfy the recurrences
	\begin{equation*}
		\alpha_c(n+1,\ell)=(c-n)\alpha_c(n,\ell)+\alpha_c(n,\ell-1), \quad
		\beta_c(n+1,\ell)=(\ell+c)\beta_c(n,\ell)+\beta_c(n,\ell-1).
	\end{equation*}
Then we have the following lemma, which is easily proven recursively.

\begin{lem}\label{lem:zetazderivative}
For a $\mathcal{C}^\infty$-function $g$, $n\in\N_0$, and $c\in \R$, we have
	\begin{align*}
		\left(\zeta^{-1}\frac\partial{\partial z}\right)^{\! n} \zeta^{c} g(z) &= \z^{c- n}\sum_{\ell=0}^n (2\pi i)^{n-\ell} \alpha_{c}(n,\ell)  \frac{\partial^\ell}{\partial z^\ell} g(z), \\
				\frac{\partial^n}{\partial z^n}\z^{c} g(z)&=\sum_{\ell=0}^{n} \lrb{2\pi i}^{n-\ell} \beta_c(n,\ell) \z^{c+\ell}\lrb{\zeta^{-1}\frac{\partial}{\partial z}}^\ell g(z).
	\end{align*}
\end{lem}
\section{Proofs of Theorems \ref{T: Main Theorem}, \ref{thm:main2}, and \ref{T:Quasimodular}} \label{Proofs}
	\subsection{Proof of Theorems \ref{T: Main Theorem} and \ref{thm:main2}}
 We first prove the results on the limiting behaviour of the sequences $\mathcal{A}_{a,k,r}$ and $\mathcal{B}_{a,k,r,s}$.
	\begin{proof}[Proof of \Cref{T: Main Theorem}]
		\begin{enumerate}[wide, labelwidth=!, labelindent=0pt]
		\item Let $\nu,\ell \in\N$, and let $\mathscr{P}(\nu,\ell)$ denote the set of partitions of $\nu$ into $\ell$ parts, and let $\mathscr{P}(\nu)$ be the set of all partitions of $\nu$. Consider
		\begin{align*}
			\mathcal V_{a,k,r}(q) :=& \ q^{-\frac{ra(a+1)}{2}} \mathcal{A}_{a,k,r}(q) \prod_{n\ge 1} {(1-q^{rn})(1-q^n)^{k}} \\
			=& \ q^{-\frac{ra(a+1)}{2}} \sum_{1\leq n_1 < n_2 < \dots < n_a} \dfrac{q^{r(n_1 + \dots + n_a)}}{\lp 1 - q^{n_1} \rp^k \cdots \lp 1 - q^{n_a} \rp^k} \prod_{n\ge 1} {(1-q^{rn})(1-q^n)^{k}} \\
		=& \prod_{n\ge 1} {(1-q^{rn})}\sum_{\nu\geq0} \sum_{\ell=0}^a\sum_{\lambda\in \mathscr{P}(\nu,\ell)} q^{r|\lambda|}\prod_{j=1}^a\dfrac{1}{(1-q^{\lambda_{a-j+1}+j})^k} \prod_{n\ge 1} {(1-q^n)^{k}} ,
		\end{align*}
		where in the last equality we change $n_j\mapsto \lambda_{a-j+1}+j$, turning $(n_1,n_2,\dots)$ into a partition $\lambda = \lp \lambda_1, \dots, \lambda_a \rp$. For convenience we allow some of the later entries to be zero.

		Now, given $\lambda$, let $m$ be the smallest integer that cannot be written of the form $\lambda_{a-j+1}+j$ for some $1 \leq j \leq a$. 
		By definition of $m$, there is an index~$n$ such that $\lambda_{a-n+1}+n=m-1$. Moreover, in that case, $\lambda_{a-n+1}=0$, because otherwise $m$ would not be the smallest such index. Hence, we have $\ell=m-1$. 
		Now, one inductively shows that $\lambda_{a-j+1}+j\geq m+j-\ell$ for all $j>\ell$, that is, $\lambda_{a-j+1}\geq m-\ell=1$. Therefore, we have
		\[
			|\lambda|+m\geq (a-\ell)+m=a+1.
		\]
		Hence, the values $\lambda_{a-j+1}+j$ take on each value from $1$ to $a-|\lambda|r$ as $j$ varies, and we have
		\[
			q^{|\lambda|r}\prod_{j=1}^k\dfrac{1}{(1-q^{\lambda_j+a-j+1})^k} \prod_{n\ge 1} {(1-q^n)^{k}}\equiv q^{|\lambda|r} \pmod {q^{a+1}}.
		\]
		This implies that
		\begin{align*}
			\mathcal V_{a,k,r}(q)\equiv & \prod_{n\ge 1} {(1-q^{rn})}\sum_{\nu\geq0} \sum_{\ell=0}^a\sum_{\lambda\in \mathscr{P}(\nu,\ell)} \! q^{|\lambda|r}
			\equiv \prod_{n\ge 1} {(1-q^{rn})}\sum_{n\geq0} \sum_{\lambda\in \mathscr{P}(n)} \! q^{r|\lambda|} \equiv 1 \pmod{ q^{a+1}},
		\end{align*}
		as desired. We note that since the range of summation and denominator of the summands for $\mathcal{B}_{a,k,r,s}$ are the same, along with the fact that the exponents on $q$ are now quadratic rather than linear, the same argument with slight modifications goes through for $\mathcal{B}_{a,k,r,s}$.
		\item From \Cref{T: NO}, we have
		\begin{align*}
			\prod_{n\ge 1} \dfrac{1}{(1-q^{rn})(1-q^n)^{k}} &= \sum_{m\ge 0} p(m) q^{rm} \sum_{\ell\ge 0} q^{\ell} \sum_{\lambda\vdash \ell}\prod_{h\in \mathcal{H}(\lambda)} \lrb{1-\dfrac{k-1}{h^2}} \\
				&= \sum_{n\ge 0} \sum_{m=0}^{\left\lfloor\frac{n}{r}\right\rfloor}p(m) \sum_{\lambda\vdash(n-mr)}\prod_{h\in \mathcal{H}(\lambda)} \lrb{1-\dfrac{k-1}{h^2}} q^n.
		\end{align*}
		Plugging this into (1) gives the result.
		\item For $r=1$, we have by \Cref{T: NO} that
		\begin{align*}
			\prod_{n\ge 1} \dfrac{1}{(1-q^{rn})(1-q^n)^{k}} &= \prod_{n\ge 1} \dfrac{1}{(1-q^n)^{k+1}} = \sum_{\ell\ge 0} \sum_{\lambda\vdash \ell}\prod_{h\in \mathcal{H}(\lambda)} \lrb{1-\dfrac{k}{h^2}} q^{\ell}.
		\end{align*}
		Again, plugging this into (1) gives the result. The argument for $\mathcal{B}_{a,k,r,s}$ is similar.  \qedhere
	\end{enumerate}
	\end{proof}
	Next, we prove \Cref{thm:main2}.
\begin{proof}[Proof of \Cref{thm:main2}]
We have
\begin{align*}
\mathcal{A}_{a,k,r}(q) &= \sum_{1\leq n_1 < n_2 < \dots < n_a} \dfrac{q^{r(n_1 + \dots + n_a)}}{\lp 1 - q^{n_1} \rp^k \cdots \lp 1 - q^{n_a} \rp^k} \\
&=  \sum_{1\leq n_1 < n_2 < \dots < n_a}\sum_{m_1,\ldots,m_a\geq 1} \prod_{j=1}^a \binom{ m_j + k  - 2}{ m_j -1} q^{(r-1)(n_1 + \dots + n_a)}q^{n_1m_1+\cdots+n_am_a} \\
&= \sum_{\ell,n\geq1} \sum_{\substack{1\leq n_1 < n_2 < \dots < n_a \\ n_1+\ldots+n_a=n}}\sum_{\substack{m_1,\ldots,m_a\geq1 \\ n_1m_1+\cdots+n_am_a=\ell}} \prod_{j=1}^a \binom{ m_j + k  - 2}{ m_j -1} q^{(r-1)(n_1 + \dots + n_a)}q^{m_1n_1+\cdots+m_an_a}.
\end{align*}
Hence, by noting that $\binom{ m_j + k - 2}{ m_j-1} $ is a polynomial in $k$ of degree $m_j-1$, one can take
\[
P_{a,\ell,n}(k) := \sum_{\substack{1\leq n_1 < n_2 < \dots < n_a \\ m_1,\ldots,m_a\geq1 \\ n_1+\ldots+n_a=n \\ m_1n_1+\cdots+m_an_a=n+\ell}} \prod_{j=1}^a \binom{m_j + k - 2}{m_j-1} .
\]
Note that $P_{a,\ell,n}$ may vanish if the system of equations in the $n_j$ and $m_j$ has no solutions (e.g., if $a=1$ and $n\nmid \ell$). Also, observe that $\sum_{j} (m_j-1)\leq \ell$, where equality is obtained if for all $m_j$ with $m_j\geq2$ one has $n_j=1$.

Similarly, one proves the result for the functions $\mathcal{B}_{a,k,r,s}$ if one fixes ${n_1^2+\ldots+n_a^2=m}$. The inequalities
$\frac{n^2}{a} \leq m\leq n^2$ with $n=n_1+\ldots+n_a$ are standard.
\end{proof}
We are now ready to prove \Cref{cor:main2}.
\begin{proof}[Proof of \Cref{cor:main2}]
One can take
\[
\mathcal{P}_{a,r,n}(k) := \sum_{\substack{\ell,m\geq 0 \\ mr+\ell=n}} P_{a,\ell,m}(k) = \sum_{\substack{1\leq n_1 < n_2 < \dots < n_a \\ m_1,\ldots,m_a\geq1 \\ m_1n_1+\cdots+m_an_a=n}} \prod_{j=1}^a \binom{m_j + k - r - 1}{m_j-r}.
\]
Now, the degree of a summand in $\mathcal{P}_{a,r,n}$ equals $\sum_{j}(m_j-r)$ if $m_j\geq r$ for all $j$ and the summand is zero otherwise.  First of all, note that since all the terms in the sum are positive there is no cancellation. Under the conditions $\sum m_jn_j =n$ and $m_j\ge r$, the sum $\sum m_j$ is maximized if the $n_j$ are as small as possible and $m_1$ is as large as possible which happens if $\lrsb{m_1,m_2,m_3,\cdots ,m_a}=\lrsb{r+n-r\frac{a(a+1)}{2},r,r,\cdots ,r}$ and $\lrsb{n_1,n_2,\cdots, n_a}=\lrsb{1,2,\cdots a}$. Hence, we get
\begin{align*}
	\mathcal{P}_{a,r,n}(k) = \frac{k^{n-\frac{ra(a+1)}{2}}}{\left(n-\frac{1}{2}a(a+1)r\right)!}  + O\left(k^{n-\frac{ra(a+1)}{2}-1}\right).
\end{align*}
The construction of the polynomial~$\mathcal{P}_{a,r,s,n}$ is analogous, and the arguments for its degree are similar as well.
\end{proof}

	\subsection{Proof of \Cref{T:Quasimodular}}
\Cref{T:Quasimodular} is a direct consequence of the following proposition.

	\begin{prop} \label{P:Quasimodularity}
	For $1\leq r\leq k$, $\mathcal{A}_{a,k,r}$ is a polynomial of degree at most $ak$ in the Eisenstein series $G_\ell$ for $\ell$ {\it even and odd}. Moreover, if $k=r$, then only Eisenstein series $G_\ell$ with $\ell$ even occur.
	\end{prop}
	\begin{proof}
		In this proof we use ideas similar to those in the proof of Theorem~1.1 in~\cite{Bac23}. Note that $\mathcal{A}_{a,k,r}$ has weight $ak$.
Recall that the quasi-shuffle algebra on $\IN^2$ with $(k,r)\diamond (\ell,s)=(k+\ell,r+s)$ is surjective onto the algebra~$\mathcal{Z}_q$. In particular, the multiplication of $q$-series corresponds to the quasi-shuffle product on~$\IN^2$. By \Cref{prop:qsh} applied to $(k,r)\in \IN^2$, we have, with $\mathcal{A}_{k,r}:=\mathcal{A}_{1,k,r}$,
\begin{equation*}
	\exp \left(\sum_{n\geq 1} \frac{(-1)^{n+1}}{n} \mathcal{A}_{nk,nr}(q) X^n \right) = \sum_{j\geq 0} \mathcal{A}_{j,k,r}(q) X^j.
\end{equation*}
In particular, $\mathcal{A}_{j,k,r}$ can be expressed as a polynomial in the $\mathcal{A}_{nk,nr}$ (with $a=1$). Now,
		\begin{equation*}
			\mathcal{A}_{k,r}(q) = \sum_{n\geq 1} \frac{q^{rn}}{(1-q^n)^{k}} = \sum_{m,n\geq1} \binom{m+k-r-1}{m-r} q^{mn}.
		\end{equation*}
	Since, $m\mapsto \binom{m+k-r-1}{m-r}$ is a polynomial of degree $k-1$, the function $\mathcal{A}_{k,r}$ is a linear combination of Eisenstein series of weight at most $k$.

		Next, assume $k=2r$. Then $\mathcal{A}_{a,2r,r}$ can be written as a polynomial in the $\mathcal{A}_{2r,r}$. Note that
		\begin{equation*}
			\mathcal{A}_{2r,r}(q) = \sum_{n\geq1} \frac{q^{nr}}{(1-q^n)^{2r}} = \sum_{m,n\geq1} \binom{m+r-1}{m-r} q^{mn},
		\end{equation*}
		where $m\mapsto\binom{m+r-1}{m-r}$ is an odd polynomial in $m$. Hence $\mathcal{A}_{2r,r}$ is a linear combination of even weight Eisenstein series, and the second part of the statement follows.
	\end{proof}

	By the above proof, one can explicitly find the expression of $\mathcal{A}_{a,k,r}$ in terms of Eisenstein series. Namely, note that
	\[
		\binom{m+k-r-1}{m-r} = \sum _{j=0}^{k-1} C_{k,r,j} m^j \quad \text{with} \quad C_{k,r,j}:=\sum _{\ell=0}^{k-1-j} \frac{\binom{\, j+\ell\,}{j} (k-r-1)^\ell s(k-1,j+\ell)}{(k-1)!},
	\]
where $s(n,\ell)$ denotes the Stirling number of the first kind (see Subsection \ref{sec:stirling}). Hence, we have
\begin{equation*}
	\mathcal{A}_{k,r}(q) = \sum_{j=0}^{k-1} C_{k,r,j} \left(\frac{B_{j+1}}{2(j+1)}+G_{j+1}(\tau)\right), \quad	\mathcal{A}_{a,k,r}(q) = \operatorname{coeff}_{\lrB{\zeta^a}}\exp \left(\sum_{n\geq 1} \frac{(-1)^{n+1}}{n} \mathcal{A}_{nk,nr}(q) \zeta^n \right) .
\end{equation*}

\subsection{Proof of \Cref{C:Congruences}}
We prove \Cref{C:Congruences} using techniques from $p$-adic modular forms.
\begin{proof}
By the proof of \Cref{P:Quasimodularity}, $\mathcal{A}_{a,2r,r}$ lies in the algebra $\mathbb{Q}\left[ G_2, G_4, G_6  \right] \cap \mathbb{Z}\llbracket q \rrbracket$ of quasimodular forms with integral coefficients, and, even more, that it can be expressed as
\begin{align*}
	\mathcal{A}_{a,2r,r} = F_0 + F_2 + \cdots + F_{2ar},
\end{align*}
where each $F_{2k}$ is a pure weight $2k$ quasimodular form with integral coefficients. We note that the only weight 0 quasimodular forms are constant functions.
	By \cite[Lemma 5.2]{AOS2023}, $E_2$ is a modular form (mod $p^b$) of weight larger than $2$ for any prime power $p^b$. Hence, $E_2$ is a modular form of positive weight (mod $m$) for any $m \geq 2$. Since $E_4,E_6$ are modular forms, it follows that each $F_{2k}$ is a modular form (mod $m$) for any $m \geq 2$. We now employ a result of Serre \cite{Ser76} on the action of Hecke operators on modular forms (see also \cite[Lemma 2.63, Theorem 2.65]{Ono2004}). Recall that the \textit{Hecke operators} $T_\ell$ for $\ell$ prime are linear operators on the space of weight~$k$ modular forms defined by the action on $f(\tau)=\sum_{n \geq 0} c(n) q^n$ by
	\begin{align*}
		T_{\ell}\!\lp f\rp\!	 (\tau) := \sum_{n \geq 0} \lp c(\ell n) + \ell^{k-1} c\lp \dfrac{n}{\ell} \rp \rp q^n,
	\end{align*}
where we set $c(\frac{n}{\ell}):=0$ if $\ell\nmid n$. The aforementioned result of Serre states that for any fixed positive integral weight modular form $f$ and any $m \geq 2$, there exists a positive proportion of primes $\ell$ such that $T_\ell(f) \equiv 0 \pmod{m}$. Note that this result also holds for any $q$-series which is congruent (mod~$m$) to a modular form of positive integral weight. Further, observe that if $T_\ell(f) \equiv 0 \pmod{m}$, then $c(\ell n)\equiv 0 \pmod{m}$ for all $n \not \equiv 0 \pmod{\ell}$, and so for $\ell \centernot | \beta$ we have $c(\ell^2 n + \ell \beta)\equiv 0 \pmod{m}$. Since for each $F_{2j}$ relations of this sort hold for infinitely many primes~$\ell$, by the Chinese remainder theorem we can construct the required infinitely many arithmetic progressions.
\end{proof}

\section{Proof of \cref{T:Completing}}\label{sec:completing}

In this section, we show the modularity of the functions $\mathcal{B}_{a,k,r,s}$. First, we prove that we can restrict to $a=1$ (recall $\mathcal{B}_{k,r,s}=\mathcal{B}_{1,k,r,s}$).
	\begin{prop} \label{P:PolynomialityB}
		The function $\mathcal{B}_{a,k,r,s}$ is a polynomial of degree at most $a$ in $\mathcal{B}_{nk,nr,ns}$ for $n\in \N$.
	\end{prop}
	\begin{proof}
The proof is similar to the proof of \Cref{P:Quasimodularity}. The main observation is that, analogously, the quasi-shuffle algebra on $\IN^3$ with $(k,r,s)\diamond (\kappa,\varrho,\sigma)=(k+\kappa,r+\varrho,s+\sigma)$ is surjective onto the algebra generated by the function
\[
\sum_{1\leq n_1<n_2<\cdots<n_a} \prod_{j=1}^a \dfrac{q^{r_jn_j^2+s_jn_j}}{(1-q^{n_j})^{k_j}}.
\]
In particular, the multiplication of the above $q$-series corresponds to the quasi-shuffle product on~$\IN^3$. \Cref{prop:qsh} then implies that
\begin{equation*}
	\exp \left(\sum_{n\geq 1} \frac{(-1)^{n+1}}{n} \mathcal{B}_{nk,nr,ns}(q) X^n \right) = \sum_{j\geq 0} \mathcal{B}_{j,k,r,s}(q) X^j,
\end{equation*}
that is, $\mathcal{B}_{a,k,r,s}$ can be expressed as a polynomial in the $\mathcal{B}_{nt,nr,ns}$.
	\end{proof}

	Next, we write the functions $\mathcal{B}_{k,r,s}$ in terms of derivatives of Appell functions $A_{2r}$ defined in~\eqref{eq:Appell}. For convenience, we let $K:=k+r-s-1$ throughout and define
	\begin{equation*}
			f_{r,K,\ell} (\tau) := \left[\frac{\partial^\ell}{\partial z^\ell}\left(\left(\zeta^{-\frac12}-\zeta^\frac12\right)\zeta^{-K}A_{2r}\left(z,-K\tau;\tau\right)\right)\right]_{z=0}.
		\end{equation*}
	Moreover, we extend the definition of $\mathcal{B}_{k,r,s}$ by allowing $s\leq-1$ (in that case $\mathcal{B}_{k,r,s}\in \mathbb{Q}[q^{-1}]\llbracket q\rrbracket$).
	\begin{lem}\label{lem:42}
		For $0\leq s\leq k$, we have
		\begin{align*}
			\mathcal{B}_{k,r,s}(q) + \mathcal{B}_{k,r,k-s}(q) &= -\frac1{k!} \sum_{\ell=0}^{k} \frac{\a_{k-s-\frac12}(k,\ell)}{(2\pi i)^{\ell}} f_{r,K,\ell}(\tau),\\
			f_{r,K,\ell} (\tau) &= -(2\pi i)^\ell{\lrb{\lrb{\frac{1}{2}+s-k}^\ell+\sum_{j={ 1}}^{\ell}j! \beta_{{\frac{1}{2}+s-k}}(\ell,j) \lrb{\mathcal B_{j,r,s-k+j}(q) + \mathcal B_{j,r,k-s}(q)}}},
		\end{align*}
where the generalized Stirling numbers $\alpha_{c}(k,\ell)$ and $\beta_{c}(\ell,j)$ are defined by \eqref{eq:defalphabeta}.
	\end{lem}
	\begin{proof}
	We define
	\begin{equation*}
		g_{k,r,s,j}(z;\t) := -\frac{1}{j!(2\pi i)^j} \left[\left(\zeta^{-1}\frac{\partial}{\partial z}\right)^j\left(\zeta^{{ k-s-\frac12}}\left(\left(\zeta^{-\frac12}-\zeta^\frac12\right)\zeta^{-K}{A}_{2r}\left(z,-K\tau;\tau\right)\right)\right)\right]_{z=0}.
	\end{equation*}
		A direct calculation gives {that for $j\in\N_0$, we have}
\[
		g_{k,r,s,j}(z;\t)= \begin{cases} -1 & \text{if}\ j=0,\\ \mathcal B_{j,r,s-k+j}(q) + \mathcal B_{j,r,k-s}(q) & \text{if}\ j>0.\end{cases}
	\]
The first part of the lemma now follows by setting $j=k$ and \Cref{lem:zetazderivative}, whereas the second part follows directly from \Cref{lem:zetazderivative} noting that $\beta_c(\ell,0)=c^\ell$.
	\end{proof}

	We next find modular completions for the $f_{r,K,\ell}$.
\begin{defn}\label{def:phi}
	For $n\geq -1$, define
	\begin{align*}
		\chi_{n}^{+}(\tau) := -\frac{(2\pi i)^n}{(n+1)!} \sum_{\substack{0\le j\le n+1\\j\equiv n+1\Pmod2}} \binom{n +1}{j}  \frac{B_{n-j+1}{\left(\frac{1}{2}\right)}}{(2\pi i)^j}f_{r,K,j}(\tau),
	\end{align*}
		where we suppress the dependence on~$k$ and~$r$ in the notation.
	Letting $\chi_{-1}^-:=0$ and similarly suppressing dependence on~$k$ and~$r$ in the notation, for $n\in\N_0$ we define
\begin{align*}
		\chi_{n}^{-}(\tau) &= \chi_{n}^{-}(\tau,\overline\tau)\\
		&:= -\frac i{2n!} \sum_{j=0}^{2r-1} \vartheta\left(\left(j-K\right)\tau+\frac12;2r\tau\right)\left[\frac{\partial^n}{\partial z^n}\left(\zeta^{-K+j}R\left(2r z+\left(K-j\right)\tau+\frac12;2r\tau,2r\overline\tau\right)\right)\right]_{z=0}.
\end{align*}
Moreover, for $n\geq -1$ we set
	\begin{align*}
		{\widehat{\chi}_{n}}\left(\tau\right) = {\widehat{\chi}_{n}}\left(\tau,\overline\tau\right) :=\,& \chi_{n}^{+}(\tau) + \chi_{n}^{-}\left(\tau\right), \qquad
		\widehat \psi_n(\tau) = \widehat \psi_n(\tau,\overline\tau) :=\,	\psi_n^+(\tau) + \psi_n^-(\tau),
	\end{align*}
	where
	\begin{equation*}
		\psi_n^+(\tau) := \sum_{j\ge0} \frac{\left(-\frac{2\pi r}v\right)^j}{j!} \chi_{n-2j}^+(\tau), \qquad \psi_n^-(\tau) := \sum_{j\ge0} \frac{\left(-\frac{2\pi r}v\right)^j}{j!} \chi_{n-2j}^-(\tau).
	\end{equation*}
\end{defn}
	\begin{rem}
		\begin{enumerate}[wide,labelwidth=!,labelindent=0pt,label=\rm(\arabic*)]
			\item Note that $f_{r,K,\ell}$, $\psi_n^+$, and $\chi_n^{+}$ may be written in terms of each other. More precisely,
			\begin{equation*}
			f_{r,K,\ell} = 2\pi i\ell!\sum_{0\leq j\leq\frac{\ell}{2}}\frac{(-1)^{j}\pi^{2j}}{(2j+1)!} \chi_{\ell-2j-1}^+
			\quad \text{and}\quad
			\chi_n^+(\tau)	=\sum_{j\geq 0} \frac{\left(\frac{2\pi r}v\right)^j}{j!}  \psi_{n-2j}^+(\tau).
			\end{equation*}
			Note that in the second formula we implicitly use that $\chi_n^+$ and $\psi_n^+$ depend on $r$ (suppressed from the notation). In other words, statements about the modularity of $\psi_n^+$, imply analogous statements about the modularity of $\mathcal{B}_{k,r,s}+\mathcal{B}_{k,r,k-s}$, $f_{r,K,\ell}$, and $\chi_n^+$.
			\item The explicit formulas for $\widehat\psi_{-1}$ and $\widehat\psi_{1}$ in \Cref{E:Psi} follows from the above definitions and~\eqref{eq:Rell}.
		\end{enumerate}
	\end{rem}
	The following theorem shows modularity properties of $\widehat\psi_n$.
	\begin{thm}\label{T:Trans}
		We have for $\pmat{a & b \\ c & d}\in\SL_{2}(\Z)$
		\begin{equation*}
			\widehat{\psi}_{n}\left(\frac{a\tau+b}{c\tau+d}\right) = (c\tau+d)^{n+1} \widehat{\psi}_{n}(\tau).
		\end{equation*}
	\end{thm}

	\begin{proof}
		Define
		\begin{equation*}
			\widehat{f}_{r,K}(z;\tau) = \widehat{f}_{r,K}(z;\tau,\overline\tau) :=\zeta^{-K} \widehat{A}_{2r}\left(z,-K\tau;\tau\right) = f_{r,K}^{+}(z;\tau) + f_{r,K}^{-}(z;\tau),
		\end{equation*}
		where
		\begin{align*}
			f_{r,K}^{+}(z;\tau) &:= \zeta^{-K} A_{2r}\left(z,-K\tau;\tau\right),\\
			f_{r,K}^-(z;\tau) &= f_{r,K}^-(z;\tau,\overline\tau) := \zeta^{-K} \left(\widehat A_{2r}\left(z,-K\tau;\tau\right)-A_{2r}\left(z,-K\tau;\tau\right)\right).
		\end{align*}
		We claim that
		\begin{equation*}
			\widehat{f}_{r,K}\left(z;\tau\right) = \sum_{n\geq-1}{ \widehat{\chi}_{n}}(\tau)z^n	, \qquad f_{r,K}^{+}	(z;\tau) = \sum_{n\ge-1} \chi_{n}^{+}(\tau) z^{n}, \qquad f_{r,K}^{-} (z;\tau) = \sum_{n\geq 0} \chi_{n}^{-}(\tau) z^{n}.
		\end{equation*}
		We have
		\begin{align*}
			&f_{r,K}^+(z;\tau) = \frac{\zeta^\frac12}{1-\zeta}\left(\zeta^{-\frac12}-\zeta^\frac12\right) \zeta^{-K} A_{2r}\!\left(z,-K\tau;\tau\right)\\
			&= -\!\sum_{a\ge-1} B_{a +1}\!\left(\frac12\right) \frac{(2\pi iz)^a}{(a+1)!} \sum_{b\ge0} \left[\frac{\partial^b}{\partial z^b}\left(\!\left(\zeta^{-\frac12}-\zeta^\frac12\right)\zeta^{-K}A_{2r}\!\left(z,-K\tau;\tau\right)\!\right)\right]_{z=0}\frac{z^b}{b!}\\
			&=-\!\sum_{a\ge-1} B_{a +1}\!\left(\frac12\right) \frac{(2\pi iz)^a}{(a+1)!} \sum_{b\ge0} f_{r,K,b}(\tau) \frac{z^b}{b!}
			= -\!\sum_{n\ge-1} (2\pi iz)^n \!\sum_{0\le j\le n+1} \frac{f_{r,K,j}(\tau)}{j!(n-j +1)!} \frac{B_{n-j +1}\!\left(\frac12\right)}{(2\pi i)^j}\\
			&= -\sum_{n\ge-1} \frac{(2\pi iz)^n}{(n +1)!} \sum_{0\le j\le n+1} \binom{n +1}{j}  \frac{B_{n-j+1}\!\left(\frac12\right)}{(2\pi i)^j}f_{r,K,j}(\tau) = \sum_{n\ge-1} \chi_{n}^{+}(\tau)z^n.
		\end{align*}
		For the final step we note that $B_n(\frac12)=0$ for $n$ odd. Thus only those $j$ survive for which $j\equiv n+1\Pmod2$. Next, we compute
		\begin{align*}
			\hspace{.5cm}&\hspace{-.5cm}f_{r,K}^-(z;\tau) = \frac{i\zeta^{-K}}2 \sum_{j=0}^{2r-1} \zeta^j \vartheta\left(-K\tau+j\tau+\frac{2r-1}2;2r\tau\right) R\left(2r z+K\tau-j\tau-\frac{2r-1}2;2r\tau\right)\\
			&= -\frac{i\zeta^{-K}}2 \sum_{j=0}^{2r-1} \zeta^j \vartheta\left(\left(j-K\right)\tau+\frac12;2r\tau\right) R\left(2r z+\left(K-j\right)\tau+\frac12;2r\tau\right)\\
			&= -\frac i2 \sum_{j=0}^{2r-1} \vartheta\left(\left(j-K\right)\tau+\frac12;2r\tau\right) \sum_{a\ge0} \left[\frac{\partial^a}{\partial z^a}\left(\zeta^{-K+j}R\left(2r z+\left(K-j\right)\tau+\frac12;2r\tau\right)\right)\right]_{ z=0} \frac{z^a}{a!} \\
			&= \sum_{n\ge0} \chi_{n}^{-}(\tau)z^{n}.
		\end{align*}
		To determine the transformation law, we rewrite $\widehat{f}_{r,K}$. By \Cref{L:Ahat} \ref{I:Ahatell},
		we obtain
		\begin{equation*}
			\widehat{f}_{r,K}(z;\tau) = \widehat{A}_{2r}(z,0;\tau).
		\end{equation*}
		In particular, $\widehat{f}_{r,K}$ is independent of $K$ and so we write $\widehat{f}_{r}$. From \Cref{L:Ahat} \ref{I:Ahatsl2z}, we have
		\begin{equation*}
			\widehat{f}_{r}\left(\frac{z}{c\tau+d};\frac{a\tau+b}{c\tau+d}\right)=\widehat{A}_{2r} \left(\frac{z}{c\tau+d},0;\frac{a\tau+b}{c\tau+d}\right) = (c\tau+d)e^{-\frac{2\pi ircz^2}{c\tau+d}} \widehat{A}_{2r}(z,0;\tau).
		\end{equation*}
		Thus $\widehat{f}_{r}$ satisfies the modular transformation law of a Jacobi form of index $-2r$.

		Note that $\widehat{f}_{r}(z;\tau)$ has a simple pole at $z=0$.
		Thus we may write
		\begin{equation*}
			\widehat{f}_{r}(z;\tau)= \sum_{n\geq-1} \chi_{n}(\bar{z};\tau)z^n
		\end{equation*}
		for certain $\chi_{n}(\overline z;\tau)=\chi_n(\overline z;\tau,\overline\tau)$. Then we have
		\begin{equation*}
			\widehat{\psi}_{n}(\tau) = \sum_{j} \frac{\left(-\frac{2\pi r}{v}\right)^{j}}{j!} \chi_{n-2j}(0;\tau).
		\end{equation*}
		Note that this is a finite sum. By Proposition 3.1 of \cite{Br}, $\widehat{\psi}_{n}$ transforms like a modular form of weight $n+1$.
	\end{proof}

By taking $\pmat{a & b \\ c & d}=\pmat{-1 & 0 \\ 0 & -1}$ in \Cref{T:Trans}, we obtain the following corollary.
\begin{cor}\label{cor:psieven}
For $n$ even, we have $\widehat{\psi}_n=0$.
\end{cor}

	The following lemma shows that we can reduce to the case $K=r$. The proof follows directly from \eqref{E:Elliptic}, \eqref{eq:multipletosingleappell}, and \eqref{eq:mu} with $\tau\mapsto\frac\tau{2r}$, $z=\frac j{2r}-\frac\tau{4r}$, $a=r-K$, and $b=0$ and \Cref{L:MuShift} with $n=K-r$, $\tau\mapsto\frac\tau{2r}$, $z_1=z$, and $z_2=\frac j{2r}-\frac\tau{4r}$.
	\begin{lem}\label{lem:K=r}
		We have
		\begin{multline*}
			A_{2r}\!\left(z,-K\tau;\tau\right)
			=\zeta^{K-r} A_{2r}\!\left(z,-r\tau;\tau\right) + \frac{(-1)^{K+r} i}{2r} q^{-\frac{(K-r)^2}{4r}- \frac{K-r}{4r}-\frac{\sgn(K-r)}{8r}-\frac{|K-r|}{4r}-\frac1{16r}} \zeta^{r-\frac{\sgn(K-r)}{8r}-\frac12} \\
			\times \sum_{j\pmod{2r}}\zeta_{4r}^{\sgn(K-r)j}\zeta_{2r}^{(K-r)j}\vartheta\!\left(\frac{-r\tau+j}{2r}+\frac{(2r-1)\tau}{4r};\frac\tau{2r}\right)\\
			\times\sum_{\ell=1}^{|K-r|} (-1)^\ell \zeta^{\sgn(K-r)\ell} \zeta_{2r}^{-\sgn(K-r)j\ell} q^{\frac{\ell\sgn(K-r)}{4r}-\frac{\ell(\ell-1)}{4r}+\frac{\ell|K-r|}{2r}}.\qedhere
		\end{multline*}
	\end{lem}

	From now on, we restrict to the case $K=r$. Then, $\widehat\psi_n$ is a completion of $\psi_n^+$.
	\begin{lem}\label{L:PsiLimit}
		We have
		\begin{equation*}
			\lim_{\overline\tau\to-i\infty} \widehat\psi_n(\tau,\overline\tau) = \psi_n^+(\tau).
		\end{equation*}
	\end{lem}
	\begin{proof}
		 First note that
		\begin{equation*}
			\lim_{\overline\tau\to-i\infty} \chi_{n-2j}^+(\tau) = \chi_{n-2j}^+(\tau), \quad \lim_{\overline\tau\to-i\infty} \psi_n^+(\tau) = \psi_n^+(\tau).
		\end{equation*}
		Therefore we are done if we show that for $\nu$ odd
		\begin{equation*}
			\lim_{\overline\tau\to-i\infty} \chi_\nu^-(\tau,\overline\tau) = 0.
		\end{equation*}
		We have
		\begin{equation*}
			\lim_{\overline\tau\to-i\infty} \chi_\nu^-(\tau,\overline\tau) = -\tfrac i{2\nu!} \sum_{j=0}^{2r-1} \vartheta\!\left((j-r)\tau+\tfrac12;2r\tau\right) \!\left[\tfrac{\partial^\nu}{\partial z^\nu}\left(\zeta^{j-r}\!\lim_{\overline\tau\to-i\infty}\!R\!\left(2r z+(r-j)\tau+\tfrac12;2r\tau\right)\!\right)\right]_{z=0}\!.
		\end{equation*}
		So it is enough to show that each summand on the right-hand side is $0$. We now write
		\begin{equation*}
			R\!\left(2r z+(r-j)\tau+\tfrac12;2r\tau\right)
			= -i\!\sum_{n\in\Z+\tfrac12}\!\left(\sgn(n)-E\!\left(\!\left(n+\tfrac{r-j}{2r}\right)2\sqrt{rv}+ \tfrac {2\sqrt{r}y}{\sqrt v}\right)\!\right) \zeta^{-2rn} q^{-rn^2+(j-r)n}.
		\end{equation*}
		We next compute
		\begin{equation*}
			\lim_{\overline\tau\to-i\infty} E\!\left(\!\left(n+\frac{r-j}{2r}\right)2\sqrt{rv}+\frac{2\sqrt{r}y}{\sqrt v}\right) = \lim_{v\to\infty} E\!\left(\!\left(n+\frac{r-j}{2r}\right)2\sqrt{rv}+\frac{2\sqrt{r}y}{\sqrt v}\right) = \sgn\left(n+\frac{r-j}{2r}\right)
		\end{equation*}
		since $\lim_{v\to\infty}E(\pm v)=\pm1$.	Recalling that $0\le j<2r$, we conclude
		\begin{equation*}
			\sgn(n) - \sgn\left(n+\frac{r-j}{2r}\right) =
			\begin{cases}
				-1 & \text{if $j=0$ and $n=-\frac12$},\\
				0 & \text{otherwise}.
			\end{cases}
		\end{equation*}
		In the case $j=0$, one checks that
		\begin{align*}
		\lim_{\overline\tau\to i\infty}R\left(2z+r\tau+\frac12;2r\tau\right) &= -i\zeta^{r} q^\frac r4.
	\end{align*}
		Now the claim easily follows.
	\end{proof}
	\begin{proof}[Proof of \Cref{T:Completing}]
		Part (1) of follows from \Cref{P:PolynomialityB}. For parts (2), (3), we use \Cref{lem:42} and \Cref{def:phi} to express $\mathcal{B}_{k,r,s} + \mathcal{B}_{k,r,k-s}$ with the functions~$\psi_n^+$. In case $K=r$, by \Cref{T:Trans} and \Cref{L:PsiLimit} these functions admit modular completions. In case $K\neq r$, \Cref{lem:K=r} guarantees that, up to theta function times powers of $q$, the functions $\psi_n^+$ admit modular completions.
	\end{proof}
	We next determine the action lowering operator $L:=-2iv^2\frac\partial{\partial \bar{\tau}}$ on $\psi_m$. It turns out that this operator relates the $\widehat\psi_m$. We define the theta function (see \cite{Sh} for their modularity properties) with congruence conditions
	\begin{equation*}
		\Theta_{a,b}(\tau) := \sum_{\substack{n\in\Z\\n\equiv a\Pmod b}} q^\frac{n^2}{2b^2}.
	\end{equation*}
	\begin{lem}\label{T:Ell}
		We have, for $m$ odd,
		\[
		L\left(\widehat\psi_{m}\right)(\tau) = 2\pi rv \widehat\psi_{m-2}(\tau)	-\frac{ir^{\frac{m}{2}}\pi^{\frac{m-1}{2}}}{2\left(\frac{m-1}{2}\right)!v^{\frac{m}{2}-1}} \sum_{j=0}^{2r-1} \left|\Theta_{r-j,2r}(2\tau)\right|^2.
		\]
	\end{lem}
	\begin{proof}
		We let
		\begin{align*}
			g(z,\overline z;\tau) &:= e^{-\frac{2\pi rz^2}v} \widehat f_r(z;\tau) = \sum_{\ell\ge0} \frac{\left(-\frac{2\pi rz^2}v\right)^\ell}{\ell!} \sum_{n\ge-1} \chi_n(\overline z;\tau) z^n = \sum_{m\ge-1} \sum_{\substack{\ell\ge0\\n\ge-1\\m=2\ell+n}}\frac{\left(-\frac{2\pi r}v\right)^\ell}{\ell!}\chi_n(\overline z;\tau) z^m.
		\end{align*}
Hence, we have
\[ g(z,0;\tau) = \sum_{m\ge-1} \widehat\psi_m(\tau) z^m.\]
		By \Cref{E:Psi}, we have
		$\widehat\psi_{-1}(\tau)=-\frac1{2\pi i}$. Thus, for $m\in2\N_{0}+1$, we have
		\begin{equation*}
			\widehat\psi_m(\tau) = \frac1{m!}\left[\frac{\partial^m}{\partial z^m}\left(g(z,0;\tau)+\frac1{2\pi i}\frac1z\right)\right]_{z=0}.
		\end{equation*}
		This yields
		\begin{equation}\label{lowP}
			L\left(\widehat\psi_m\right)(\tau) = \frac1{m!}\left[\frac{\partial^m}{\partial z^m}L(g(z,0;\tau))\right]_{z=0}.
		\end{equation}
		Let $f(z,\overline z;\tau):=\widehat f_r(z;\tau)$. Then
		\begin{equation*}
			L(g(z,0;\tau)) = L\left(e^{-\frac{2\pi rz^2}v}\right) f(z,0;\tau) + e^{-\frac{2\pi rz^2}v} L(f)(z,0;\tau).
		\end{equation*}
		We compute
		\begin{equation*}
			L\left(e^{-\frac{2\pi rz^2}v}\right) = {2\pi rz^2} e^{-\frac{2\pi rz^2}v}.
		\end{equation*}
		This term contributes to \eqref{lowP} as
		\begin{equation*}
			{ \frac{2\pi r}{m!}}\left[\frac{\partial^m}{\partial z^m}\left(z^2g(z,0;\tau)\right)\right]_{z=0} = \frac{2\pi r}{m!}\left[\frac{\partial^m}{\partial z^m}\left(z^2\left(g(z,0;\tau)+\frac1{2\pi iz}\right)-\frac z{2\pi i}\right)\right]_{z=0}.
		\end{equation*}
		Now $g(z,0;\tau)+\frac1{2\pi iz}$ does not have a pole in $z=0$. So we need to differentiate $z^2$ twice. If $m=1$, then we get
		\begin{equation*}
			ir = 2\pi r\widehat{\psi}_{-1}(\tau).
		\end{equation*}
		If $m\ge 2$, then we obtain
		\begin{equation*}
			\frac{4\pi r}{m!} \binom m2 \frac{\partial^{m-2}}{\partial z^{m-2}}\left(g(z,0;\tau)+\frac1{2\pi iz}\right) = { 2\pi r} \widehat\psi_{m-2}(\tau).
		\end{equation*}
		We next compute
		\begin{equation*}
			L\left(\widehat f_r(z;\tau)\right) = L\left(\widehat{A}_{2r}(z,0;\tau)\right) =-v^2\sum_{j=0}^{2r-1} \zeta^j\vartheta\left(j\tau+\frac{1}{2};2r\tau\right) \frac{\partial}{\partial \bar{\tau}} R\left(2rz-j\tau+\frac{1}{2};2r\tau\right).
		\end{equation*}
		Making the change of variables $n\mapsto n+\frac j{2r}$, we obtain
		\begin{align*}
			\frac{\partial}{\partial \bar{\tau}} R\left(2rz-j\tau+\frac{1}{2};2r\tau\right) &= iq^{\frac{j^2}{4r}}\zeta^{-j} \sum_{n\in\Z+\frac{1}{2}-\frac{j}{2r}}q^{-rn^2}\zeta^{-2rn} \frac{\partial}{\partial\bar{\tau}} E\left(\left(n+\frac{y}{v}\right)2\sqrt{rv}\right)\\
			&= -q^{\frac{j^2}{4r}}\zeta^{-j} \sqrt{r}\sum_{n\in\Z+\frac{1}{2}-\frac{j}{2r}} q^{-rn^2} \zeta^{-2rn} e^{-4\pi\left(n+\frac{y}{n}\right)^2rv}\left(\frac{n}{\sqrt{v}}-\frac{y}{v^{\frac{3}{2}}}\right),
		\end{align*}
		using that $E'(x)=2e^{-\pi x^2}$. We now simplify
		\begin{equation*}
			q^{-rn^2} \zeta^{-2rn} e^{-4\pi\left(n+\frac{y}{v}\right)^2rv} = e^{-2\pi irn^2\bar{\tau} - \frac{4\pi ry^2}{v}},
		\end{equation*}
		noting that $\bar{z}=0$. Thus we have
		\begin{equation*}
			\frac{\partial}{\partial \bar{\tau}} R\left(2rz-j\tau+\frac{1}{2};2r\tau\right) = -\frac1{\sqrt v}q^{\frac{j^2}{4r}}\zeta^{-j} e^{- \frac{4\pi ry^2}{v}} \sqrt r \sum_{n\in\Z+\frac{1}{2}-\frac{j}{2r}} \left(n-\frac{y}{v}\right)e^{-2\pi irn^2\bar{\tau}}.
		\end{equation*}
		This yields
		\begin{align*}
			L\left(\widehat{f}_{r}(z;\tau)\right) = v^{\frac{3}{2}} \sqrt re^{- \frac{4\pi ry^2}{v}} \sum_{j=0}^{2r-1} q^{\frac{j^2}{4r}}\vartheta\left(j\tau+\frac{1}{2};2r\tau\right) \sum_{n\in\Z+\frac{1}{2}-\frac{j}{2r}}\left(n-\frac{y}{v}\right)e^{-2\pi irn^2\bar{\tau}}.
		\end{align*}

		This term contributes to \eqref{lowP} as
		\begin{equation}\label{Cont2}
			\frac{v^{\frac32}\sqrt{r}}{m!} \sum_{j=0}^{2r-1} q^{\frac{j^2}{4r}}\vartheta\left(j\tau+\frac12;2r\tau\right) \sum_{n\in\Z+\frac{1}{2}-\frac{j}{2r}} e^{-2\pi irn^2\bar{\tau}} \left[\frac{\partial^{m}}{\partial z^m} \left(e^{-\frac{4\pi ry^2}{v}}\left(n-\frac{y}{v}\right)\right)\right]_{z=0}.
		\end{equation}
		We are left to compute
		\begin{align*}
			\left[\frac{\partial^{m}}{\partial z^m} \left(e^{-\frac{4\pi ry^2}{v}}\left(n-\frac{y}{v}\right)\right)\right]_{z=0}\hspace{-0.2cm} = m \left[\frac{\partial}{\partial z}\left(n-\frac{y}{v}\right)\right]_{z=0} \left[\frac{\partial^{m-1}}{\partial z^{m-1}}e^{-\frac{4\pi ry^2}{v}}\right]_{z=0}\hspace{-0.2cm} = m \frac{i}{2v} \frac{\left(\frac{\pi r}{v}\right)^{\frac{m-1}{2}}}{\left(\frac{m-1}{2}\right)!} (m-1)!,
		\end{align*}
		using that $m$ is odd and that as $\bar z=0$, we have $z=2iy$.
		Thus \eqref{Cont2} becomes
		\begin{equation*}
			\frac{i r^{\frac{m}{2}}\pi^{\frac{m-1}{2}}}{2 \left(\frac{m-1}{2}\right)! v^{\frac{m}{2}-1}} \sum_{j=0}^{2r-1} q^{\frac{j^2}{4r}}\vartheta\!\left(j\tau+\frac12;2r\tau\right) \sum_{n\in\Z+\frac{1}{2}-\frac{j}{2r}} e^{-2\pi irn^2 \bar{\tau}}.
		\end{equation*}
		The claim follows by simplifying and writing
		\begin{equation*}
			\vartheta\!\left(j\tau+\frac12;2r\tau\right) =-q^{-\frac{j^2}{4r}} \sum_{n\in\Z+\frac{1}{2}-\frac{j}{2r}} q^{rn^2}.\qedhere
		\end{equation*}
	\end{proof}

	\section{Variations and open questions}\label{sec:open}

\subsection{Variations going back to MacMahon}\label{rem:variations}

	In MacMahon's original work \cite{M1921}, he considered many modifications of $\mathcal{A}_a$. Here, we mention three variations and indicate how the proofs of the relevant theorems would generalize.
	First of all, consider the complementary function
	\begin{align*}
		\mathcal{A}^*_{a,k,r}(q) := \sum_{1\leq n_1 \leq n_2 \leq \cdots \leq n_a} \dfrac{q^{r(n_1 + \dots + n_a)}}{\lp 1 - q^{n_1} \rp^k \cdots \lp 1 - q^{n_a} \rp^k}.
	\end{align*}
	These analogues of $\mathcal{A}_{a,k,r}$ are important in the literature on MacMahon's $q$-series \cite{AOS2023} and are connected to $q$-multiple zeta star values \cite{PP15}. Note that $\mathcal{A}_{1,k,r}=\mathcal{A}^*_{1,k,r}$ for all $k,r$. More generally, a similar proof as the above yields
	\begin{align}\label{eq:A*}
		\mathcal{A}^*_{a,k,r}(q) = \operatorname{coeff}_{\lrB{\zeta^a}}\exp \left(\sum_{n\geq 1}  \mathcal A_{1,nk,nr}(q) \frac{\zeta^n}{n} \right).
	\end{align}
	In particular, $\mathcal{A}^*_{a,2r,r}$ is a quasimodular form of weight $2ar$.

	Secondly, following MacMahon, one could additionally require that the integers $n_1, n_2, \dots, n_a$ in the definition of $\mathcal{A}_{a,k,r}$ belong to certain arithmetic progressions, e.g., enforcing that each $n_j$ should be odd, or should be $\pm 1 \pmod{5}$. We note that the argument in \Cref{P:Quasimodularity} generalizes to these cases, and produces quasimodular forms on congruence subgroups rather than on the full modular group. More generally, for $b,N\in \mathbb{N}$, consider the function
	\begin{align*}
		\sum_{\substack{1\leq n_1 < n_2 < \dots < n_a \\ n_j\equiv \pm b\pmod N}}\dfrac{q^{r(n_1 + \dots + n_a)}}{\lp 1 -  q^{n_1} \rp^{k} \cdots \lp 1 -  q^{n_a} \rp^{k}}.
	\end{align*}
	Upon substituting $k=2r$ and $q\mapsto e^{\frac{2\pi i \tau}{N}}$, this function is a quasimodular form for $\Gamma(N)$, which follows from \Cref{prop:qsh} and the fact that for $a=1$ it is a linear combination of Eisenstein series for $\Gamma(N)$. The special case $k=2r=2$ also follows from \cite[Theorem~1.9]{Ros15}.

	Lastly, for non-zero $P\in \mathbb{N}_0[x]$, consider the function
	\begin{equation*}
		\mathcal{C}_{a,k}(P;q):=
		\sum_{1\leq n_1 < n_2 < \dots < n_a} \dfrac{q^{\sum_{j=1}^a P(n_j)}}{\lp 1 - q^{n_1} \rp^k \cdots \lp 1 - q^{n_a} \rp^k}.
	\end{equation*}
	If the degree of $P$ is at least $2$, a similar argument as in the proof of \Cref{T: Main Theorem} shows that
	\begin{align*}
		q^{-\sum_{j=1}^a P(j)} \mathcal{C}_{a,k}(P;q) = \prod_{n\ge 1} \dfrac{1}{(1-q^n)^k} + O\lrb{q^{a+1}}.
	\end{align*}
	We do not expect any modular properties for $\mathcal{C}_{a,k}(P;\cdot)$ if the degree of $P$ is at least $3$.


	\subsection{Congruence properties}
	As proved in \Cref{C:Congruences}, there are many congruences for certain examples of the functions $\mathcal{A}_{a,k,r}$. In practice, the congruences generated by this particular method may be difficult to find and might lie only in very sparse arithmetic progressions. Based on computational data, we conjecture the following different family of congruences hold.
	For $n\in\IN$ and a prime $p$, we write $\nu_p(n)$ for the $p$-adic valuation of $n$, i.e., the exponent to which $p$ appears in the prime factorization of $n$. 
	\begin{conj}\label{conj:congr}
		Let $a,k,r,s \in\N$. For any prime $p$, we have
		\begin{align*}
			c_{a,k,r}\!\lp p^{\alpha+1}n + p^\alpha \beta \rp \equiv 0 \pmod{p^{\nu_p(k)-\alpha}}
		\end{align*}
		for any $0 \leq \alpha \leq \nu_p(\gcd(k,r))-1$ and $1 \leq \beta \leq p-1$. We also have
		\begin{align*}
			d_{a,k,r,s}\!\lp p^{\alpha+1}n + p^\alpha \beta \rp \equiv 0 \pmod{p^{\nu_p(k)-\alpha}}
		\end{align*}
		for any $0 \leq \alpha \leq \nu_p(\gcd(k,r,s))-1$  and $1 \leq \beta \leq p-1$.
	\end{conj}

	\begin{rem}
Note that only some of the cases covered by this conjecture (namely $c_{a,k,r}$ with $k = 2r$) are associated to quasimodular forms. In that case, results on vanishing of Hecke operators should suffice to prove this conjecture. It is of particular interest that these ``Hecke-type'' congruences appear to hold even if $\mathcal{A}_{a,k,r}$ contains non-modular odd weight Eisenstein series, as well as for $\mathcal{B}_{a,k,r,s}$ for which there exist non-holomorphic completions if $s=\frac k2$. Data also suggests many congruences exist beyond these families, and these sporadic examples appear to exist for both modular and non-modular examples. For example, for $a\in\N$, we conjecture that
		\begin{align*}
			 c_{3a,4,2} \lp 3n+2 \rp & \equiv 0 \pmod{3}, \quad   c_{3a+1,4,2} \lp 3n+2 \rp  \equiv 0 \pmod{3}, \quad c_{2,4,2}\lp 37n\rp \equiv 0 \pmod{19},\\
			 &c_{1,3,1}\lp 8n+4 \rp \equiv 0 \pmod{7},\quad c_{1,5,2}\lp 9n+1 \rp \equiv 0 \pmod{3}.
		\end{align*}
If additionally $r$ and $s$ are both odd, then we conjecture that
\begin{align*}
 d_{a,k,r,s}(2n+1) \equiv 1 \pmod{2^{\nu_p(k)}}.
\end{align*}
	\end{rem}

	\subsection{Expanding eta quotients in MacMahon-like $q$-series}
 Our study of~$\mathcal{A}_{a,k,r}$ and~$\mathcal{B}_{a,k,r,s}$ was initially motivated by Theorem~\ref{T:AOS}, which has recently been further developed. Ono and Singh \cite{OS2024} refined this result using the Jacobi triple product formula in order to prove the exact formula
	\begin{align*}
		\sum_{n \geq 0} p_3(n) q^n = q^{-\frac{k(k+1)}{2}} \sum_{m \geq k} \binom{2m+1}{m+k+1} \mathcal A_{m,2,1}(q).
	\end{align*}
	Thus, the 3-colored partition function is an infinite linear combination of quasimodular forms. In light of Theorem~\ref{T: Main Theorem}, it is natural to ask whether the $q$-series $\mathcal{A}_{a,k,r}$ and $\mathcal{B}_{a,k,r,s}$ can be used to build an exact formula for the eta quotient they approximate.

	\subsection{Hypergeometric representation and recursion}

	Previous literature on these functions also hints strongly towards deeper connections between generalized sum-of-divisor functions and basic $q$-hypergeometric functions. In order to discuss this topic, recall the \textit{$q$-Pochhammer symbol} $(a;q)_n:=\prod_{j=0}^{n-1}(1-a q^j)$ for $n\in\N_0 \cup\{\infty\}$, as well as the \textit{$q$-binomial coefficient} $\binom{n}{m}_{\!q}:=\frac{(q;q)_n}{(q;q)_{n-m} (q;q)_m}$. For basic facts about these $q$-series and their relationship to partitions, see \cite{And18}. There are several isolated results in the literature which connect generalized sum-of-divisor functions to basic $q$-hypergeometric functions. The earliest such result of which the authors are aware is due to Dilcher \cite[Corollary~2]{Dil95}, who proved that
	\begin{align*}
		\mathcal{A}_{a,1,1}^*(q) = \sum_{n \geq a} \binom{n}{a} q^n \lp q^{n+1}; q^n \rp_\infty.
	\end{align*}
Here, $\mathcal{A}^*_{a,k,r}$ is defined by \eqref{eq:A*}. More recently, Amdeberhan--Andrews--Tauraso \cite[Corollary~4.1]{AAT2023} showed that
	\begin{align*}
		\mathcal{A}^*_{a,2,1}(q) = \sum_{n \geq 1} \dfrac{(-1)^n \lp 1 + q^n \rp q^{\frac{n(n+1)}{2} + (a-1)n}}{\lp 1 - q^{2n} \rp^{2a}}.
	\end{align*}
	Both these identities have underlying finite versions; we focus on Dilcher's for simplicity. Dilcher showed \cite[Theorem~4]{Dil95} that
	\begin{align*}
		\sum_{1 \leq n_1 \leq \dots \leq n_a \leq n} \dfrac{q^{n_1 + \dots + n_a}}{\lp 1 - q^{n_1} \rp \cdots \lp 1 - q^{n_a} \rp} = \sum_{m=1}^n \dfrac{(-1)^{m+1} q^{\frac{m(m+1)}{2} + (a-1)m}}{\lp 1 - q^m \rp^a} \binom{n}{m}_{\!q} =: W_{a,n}(q).
	\end{align*}
	The proof is by double induction on $a$ and $n$, the critical ingredient of which is the observation that the identity $\binom{n}{m}_{\!q} - \binom{n-1}{m}_{\!q} = q^{n-m} \frac{1-q^m}{1-q^n} \binom{n}{m}_{\!q}$ for $q$-binomial coefficients implies that
	\begin{align*}
		W_{a,n}(q) - W_{a,n-1}(q) = \sum_{m=1}^n \dfrac{(-1)^{m+1} q^{\frac{m(m+1)}{2} + (a-1)m}}{\lp 1 - q^m \rp^a} \lp \binom{n}{m}_q - \binom{n-1}{m}_{\!q} \rp = \dfrac{q^n}{1 - q^n} W_{a-1,n}(q).
	\end{align*}
	Dilcher's finite theorem follows directly from this recurrence and the fact that the generalized sum-of-divisor functions satisfy the same recurrence and have the same initial conditions. Similar techniques and results lie behind the aforementioned results of \cite{AAT2023}, and in \cite{PP15} the same result was proved from the point of view of $q$-multiple zeta values. We also note here the result of Andrews and Rose \cite[Corollary 2]{AR2013}, which states in our notation that
	\begin{align} \label{AR Equation}
		\mathcal{A}_{a,2,1}(q) = \dfrac{(-1)^a}{\lp 2a+1 \rp! \lp q;q \rp_\infty^3} \sum_{n\geq0} (-1)^n (2n+1) \dfrac{(n+k)!}{(n-k)!} q^{\frac{n(n+1)}{2}}.
	\end{align}
	Although this formula takes a very different shape from those of Dilcher and Amdeberhan--Andrews--Tauraso, it fits broadly into the same framework. We propose that the connection between these $q$-multiple zeta values and basic $q$-hypergeometric functions should lie deeper than these examples. Using the same trick as in Dilcher's work, it is easy to derive the following lemma for
	\begin{align*}
		\mathcal{A}_{a,k,r}(n;q) &:= \sum_{1\leq n_1 < n_2 < \cdots < n_a \leq n} \dfrac{q^{r(n_1 + n_2 + \dots + n_a)}}{\lp 1 - q^{n_1} \rp^k \lp 1 - q^{n_2} \rp^k \cdots \lp 1 - q^{n_a} \rp^k}, \\
		A^*_{a,k,r}(n;q) &:= \sum_{1\leq n_1 \leq n_2 \leq \cdots \leq n_a \leq n} \dfrac{q^{r(n_1 + n_2 + \dots + n_a)}}{\lp 1 - q^{n_1} \rp^k \lp 1 - q^{n_2} \rp^k \cdots \lp 1 - q^{n_a} \rp^k}.
	\end{align*}
	\begin{lem}
		We have the recurrences
		\begin{align*}
			\mathcal{A}_{a,k,r}(n;q) - \mathcal{A}_{a,k,r}(n-1;q) &= \dfrac{q^{rn}}{\lp 1 - q^n \rp^k} \mathcal{A}_{a-1,k,r}(n-1;q), \\
			\mathcal{A}^*_{a,k,r}(n;q) - \mathcal{A}_{a,k,r}(n-1;q) &= \dfrac{q^{rn}}{\lp 1 - q^n \rp^k} \mathcal{A}_{a-1,k,r}(n;q).
		\end{align*}
	\end{lem}
	These recurrences strongly suggest that $\mathcal{A}_{a,k,r}(n;q)$ and $\mathcal{A}_{a,k,r}^*(n;q)$ can be represented as truncations of basic $q$-hypergeometric function as in \cite{AAT2023,Dil95,PP15}, in which case it would follow by taking $n \to \infty$ that $\mathcal{A}_{a,k,r}(q)$ and $\mathcal{A}_{a,k,r}^*(q)$ are basic $q$-hypergeometric functions. We also note that Andrews and Rose identified a different sort of recursion for $\mathcal{A}_{a,2,1}$ \cite[Corollary 3]{AR2013}. So far, we have been unable to produce suitable candidates that satisfy these recurrences, and leave the identification of such $q$-series as an open problem.

	In connection with \eqref{AR Equation}, we provide a slightly more precise speculation. The authors have performed numerical computations that suggest that \Cref{T: Main Theorem} (1) can be extended in the following way modulo higher powers of $q$, e.g.,
	\allowdisplaybreaks
	\begin{align*}
		q^{- a(a+1)} \mathcal{A}_{a,k,2}(q)- \dfrac{1}{\left(q^2;q^2\right)_\infty\left(q;q\right)_{\infty}^{k}} &=- k q^{a+1}\lrb{ \dfrac{\left(q^2;q^2\right)_{\infty}^3}{\left(q;q\right)_{\infty}^{k+2}}+O\left(q^{a+2}\right)},  \\
		q^{- \frac{3a(a+1)}{2}} \mathcal{A}_{a,k,3}(q)- \dfrac{1}{\left(q^3;q^3\right)_{\infty}\left(q;q\right)_{\infty}^{k}} &=- k q^{a+1}\lrb{ \dfrac{\left(q^3;q^3\right)_{\infty}^2}{\left(q;q\right)_{\infty}^{k+1}}+O\left(q^{a+2}\right)}, \\
		q^{-2a(a+1)}\mathcal{A}_{a,k,4}(q)- \dfrac{1}{\left(q^4;q^4\right)_{\infty}\left(q;q\right)_{\infty}^{k}} &=- k q^{a+1}\lrb{ \dfrac{\left(q^2;q^2\right)_{\infty}\left(q^4;q^4\right)_{\infty}}{\left(q;q\right)_{\infty}^{k+1}}+O\left(q^{a+2}\right)},  \\
		q^{- \frac{5a(a+1)}{2}}\mathcal{A}_{a,k,5}(q)- \dfrac{1}{\left(q^5;q^5\right)_{\infty}\left(q;q\right)_{\infty}^{k}} &=- k q^{a+1}\lrb{ \dfrac{\left(q^2;q^2\right)_{\infty}\left(q^3;q^3\right)_{\infty}\left(q^5;q^5\right)_{\infty}^2}{\left(q;q\right)_{\infty}^{k+1}\left(q^4;q^4\right)_{\infty}}+O\left(q^{a+2}\right)}.
	\end{align*}
	This suggests that the functions $\mathcal{A}_{a,k,r}$ might possess a representation in the form
	\begin{align*}
		\mathcal{A}_{a,k,r}(q) = \sum_{m \geq a} c_{a,k,r}(m) \varphi_{k,r,m}(q) q^{\frac{m(m+1)}{2}+(r-1)\frac{a(a+1)}{2}}
	\end{align*}
	for certain $c_{a,k,r}(m) \in \mathbb{Q}$ and $\varphi_{k,r,m}(q)$ a ``nice'' $q$-series (e.g., an eta quotient).

\subsection{Jacobi forms and quasimodular forms}

One of the basic results in the theory of Jacobi forms is that the Taylor coefficients of a Jacobi form $f\lp z;\tau\rp$ yield a sequence of quasimodular forms of growing weight \cite[equation (6) p.~31]{EZ85}. Thus, given a natural sequence of quasimodular forms of growing weight, it is natural to ask whether they form the Taylor coefficients of some Jacobi form, and whether the family can be better understood in the context of the Jacobi structure. For the family $\mathcal{A}_{a,2,1}$ this was accomplished by Rose \cite[Theorem~1.9]{Ros15} (with $n=1$, $S=\{1\}$, and $s=|S|-1=0$)
\[
	\sum_{a\geq 0} (-1)^a \mathcal{A}_{a,2,1}(q) \left(\zeta^\frac12-\zeta^{-\frac12}\right)^{2a+1} = -i \frac{\vartheta(z;\tau)}{\eta(\tau)^3}.
\]
Note that \cite[Theorem~1]{Bac23} is the same result, obtained using the theory of quasi-shuffle algebras. This is clear by the observation $\vartheta'(0;\tau) = -2\pi \eta(\tau)^3$ and the expansion
\[ \frac{\vartheta(z;\tau)}{\vartheta'(0;\tau)} =  2\pi i z \exp\left(\sum_{m\geq 1} \frac{B_{2m}}{2m} E_{2m}(\tau) \frac{(2\pi iz)^{2m}}{(2m)!}\right).
\]

Rose's work contains many examples generalizing this (including all of MacMachon's examples in~\cite{M1921}). 
His result is of more than theoretical interest; it allows him to isolate the pure weight factors of $\mathcal{A}_{a,2,1}$ \cite[Theorem~1.12]{Ros15}. It would be natural to ask whether the quasimodular forms~$\mathcal{A}_{a,2r,r}$ also form the Taylor coefficients of a Jacobi form, and if so, whether there is a natural representation of this Jacobi form as a product of theta functions. Such a representation could be useful in addressing the open problems and conjectures posed above.

\end{document}